\documentclass[fleqn,a4paper,10pt]{amsart}

\usepackage{graphicx}
\usepackage{tikz}
\usetikzlibrary{shapes,decorations.pathmorphing}
\usepackage{amsmath}
\usepackage{amssymb,amsthm}
\usepackage{comment}
\usepackage{hyperref}
\usepackage{mathtools}

\def\cF{\mathcal{F}}
\def\cP{\mathcal{P}}
\def\cU{\mathcal{U}}
\def\cR{\mathcal{R}}

\def\hr{\hat{r}}
\def\hq{\hat{q}}
\def\Sb{S_{\bullet}}
\def\Sw{S_{\circ}}
\def\Eb{E_{\bullet}}
\def\Ew{E_{\circ}}
\def\tT{\tilde{T}}

\def\cM{\mathcal{M}}
\def\cQ{\mathcal{Q}}

\newtheorem{theo}{Theorem}
\newtheorem{prop}{Proposition}
\newtheorem{lem}{Lemma}

\theoremstyle{definition}

\def\cM{\mathcal{M}}

\newcommand*\mywidebar[1]{%
   \hbox{%
     \vbox{%
       \hrule height 0.5pt 
       \kern0.5ex
       \hbox{%
         \kern-0.1em
         \ensuremath{#1}%
         \kern-0.1em
       }%
     }%
   }%
}

\title[Maps of unfixed genus and blossoming trees]{Maps of unfixed genus and blossoming trees}
\author{\'Eric Fusy}
\address{CNRS/LIX, \'Ecole Polytechnique, 91120 Palaiseau, France.}
\email{fusy@lix.polytechnique.fr}

\author{Emmanuel Guitter}
\address{Universit\'e Paris-Saclay, CNRS, CEA, Institut de physique th\'eorique, 91191, Gif-sur-Yvette, France.}
\email{emmanuel.guitter@ipht.fr}

\begin{document}
\maketitle

\begin{abstract} 
We introduce bijections between families of rooted maps with unfixed genus and families of so-called blossoming trees endowed with an arbitrary forward matching of their leaves.  We first focus on Eulerian maps with controlled vertex degrees. The mapping from blossoming trees to maps is a generalization to unfixed genus of Schaeffer's closing construction for planar Eulerian maps. 
The inverse mapping relies on the existence of canonical orientations which allow to equip the maps with canonical spanning trees, as proved by Bernardi. 
Our bijection gives in particular (here in the Eulerian case) a combinatorial explanation to the striking similarity between the (infinite) recursive system of equations which determines the partition function of maps with unfixed genus (as obtained via matrix models and orthogonal polynomials) and that determining the partition function of planar maps. All the functions in the recursive system get a combinatorial interpretation as generating functions for maps endowed with particular multiple markings of their edges. This allows us in particular to give a combinatorial proof of some differential identities satisfied by these functions. We also consider face-colored Eulerian maps
with unfixed genus and derive some striking identities between their generating functions and those of properly weighted marked maps. The same 
methodology is then applied to deal with $m$-regular bipartite maps with unfixed genus, leading to similar results. The case of cubic maps is also briefly discussed. 
\end{abstract}

\section{Introduction}
\label{sec:introduction}

\subsection{Aim of the paper}

The enumeration of maps, i.e.\ cellular embeddings of graphs into surfaces, has been a constant subject of investigation since the 
seminal papers of Tutte in the 60's \cite{Tutte62a,Tutte62b,Tutte63}, with numerous combinatorial and probabilistic results coming from various counting techniques
developed over the years \cite{PDF06,Bo08,Mi14,bona15,ey16}. An instructive exercise consists in finding connections between the different enumeration approaches 
as it may lead to a better understanding of the underlying common combinatorial objects. In this spirit, it was observed a long time ago \cite{BDG02} 
that two among the various map enumeration techniques, even though unrelated a priori, present some striking and yet unexplained similarities. 

The first approach is that of random matrix integrals, which gives access to generating functions for \emph{maps with unfixed genus}.
More precisely, performing integrals over Hermitian matrices of size $N\times N$ allows to control the genus $h$ of the maps by assigning 
them a weight $N^{2-2h}$ \cite{PDF06,DGZ95}, or alternatively to control their number $F$ of faces by assigning them a weight $N^F$ \cite{LZ04}. 
The second and totally different approach relies on bijections between maps and the so-called \emph{blossoming trees}~\cite{Sc97,Sc98,BDG02,BS02,PS06,AlPo15},
which is particularly adapted to the enumeration of \emph{planar maps}, for which the associated blossoming trees are genuine plane trees
with some decorations.

Plane trees are recursive objects by essence and their generating functions may therefore in general be computed recursively.
In the case of blossoming trees associated with standard families of rooted planar maps (for instance maps with fixed vertex degrees, bipartite maps, constellations~...), 
their generating function $R_1$ may be obtained as the first term in a family of functions $(R_i)_{i\geq 1}$ entirely determined by an infinite recursive 
system of non-linear equations. For instance, in the case rooted 4-valent planar maps (with a weight $g$ per vertex) it reads (with the convention $R_0=0$)
\begin{equation*}
R_i=1+g\ \!R_i(R_{i-1}+R_i+R_{i+1}),\ \ i\geq 1.
\end{equation*}
For many families of maps, the associated recursive system turns out to be integrable and very explicit expressions~\cite{BDG03,BG12} can be obtained for 
the functions $R_i$ (this holds especially for blossoming trees in bijection with maps of bounded vertex degrees). {}From  bijections between trees and 
maps~\cite{BDG03}, the $R_i$'s also have a direct interpretation as map generating functions\footnote{Another bijection, between planar maps and so-called labelled mobiles~\cite{BDG04}, leads to a very similar interpretation; however we will not use it in our present study.}: $R_i$ enumerates two-leg planar maps with the two legs at distance less than~$i$. 
In particular, the first term $R_1$ in the family is precisely the generating function for rooted planar maps.

A very similar integrable structure emerges in the matrix model formalism when implemented via the so-called orthogonal polynomial technique. 
There a recursive system of non-linear equations is encountered for an infinite set $(r_i)_{i\geq 1}$ of formal power series which come up as ratios of 
norms for the successive orthogonal polynomials. In particular, apart from the first term $r_1$ which enumerates rooted maps with unfixed genus,
these formal series have no direct interpretation as map generating functions and  their index $i$ has no clear combinatorial meaning. 

Remarkably, for all the standard families of maps studied so far, it may be checked that the recursive system for the $r_i$ is identical to that for the $R_i$, up to
a simple elementary (yet crucial) modification, which is moreover the same for all map families. This ``universal'' modification allows 
in particular to interpret the $r_i$'s as counting series for 
the \emph{same} blossoming trees as those counted by the $R_i$'s, with a simple slight difference in their weighting. Leaves of the trees may be classified 
by their \emph{height} $h$ (also sometimes called ``depth"): leaves of height $h$ receive a weight $i+h-1$ in $r_i$ instead of the trivial weight $1$ in $R_i$,
leading to a first contribution of the trivial tree (made of a single leaf at height $1$) to $r_i$ equal to $i$ instead of $1$ for $R_i$. For instance, in the case of rooted 4-valent maps (with a weight $g$ per vertex) the system reads (with again the convention $r_0=0$)
\begin{equation*}
r_i=i+g\ \!r_i(r_{i-1}+r_i+r_{i+1}),\ \ i\geq 1.
\end{equation*}
Even though this striking resemblance between $r_i$ and $R_i$ was quoted many years ago, it found no combinatorial explanation so far in terms
of maps and this observation therefore raises a number of questions: what is, if any, the combinatorial interpretation of the new weight $i+h-1$ for the blossoming trees enumerated by $r_i$? How does it alter
the counting of the blossoming trees, which are planar objects, so as to create generating functions for maps with unfixed genus? Can we settle a 
direct bijection between (possibly decorated) blossoming trees and maps with unfixed genus? 

\medskip
The purpose of this paper is to answer some of these questions. 
More precisely, we focus here on two main families of maps: that of \emph{Eulerian maps}, which are maps where all vertices have even degrees, and that of
\emph{$m$-regular bipartite maps},
where all vertices have degree $m$ and are bicolored with no two adjacent vertices of the same color.
We show that the extra weight~$i+h-1$ per leaf of height $h$ when passing from $R_i$ to $r_i$ gets a natural explanation if we
equip the blossoming trees with arbitrary forward matchings of their leaves (instead of planar forward matchings in the planar case). 
We then establish a bijection between blossoming trees
endowed with an arbitrary forward matching and (rooted) maps with unfixed genus (Theorems \ref{thm:eul} and \ref{thm:bip}). The passage from
blossoming trees to maps may be viewed as a simple generalization of the closing constructions of \cite{Sc97,BoSc00} in the planar case. The inverse mapping
from maps to blossoming trees is more involved and relies crucially on the existence of canonical orientations on the maps, as proved by Bernardi~\cite{Be08}, which allows to equip these maps with a canonical
spanning tree.  
The above bijection between blossoming trees and maps allows us to explain combinatorially why $r_1$, 
which has a clear blossoming tree combinatorial interpretation, is the generating function for (rooted) maps with unfixed genus, hence to give a combinatorial 
interpretation to the matrix model result. 

The combinatorial interpretation of $r_i$ for $i>1$ in terms of maps is more subtle and involves the notion of \emph{marked maps}, where some of the edges are
marked and oriented, with some restrictions on the markings. We show that these maps are indeed in bijection with blossoming trees equipped with specific marked 
pairings, enumerated by $r_i$. 

Another outcome of our comparison between the matrix integral and blossoming tree approaches is a better understanding of
so called \emph{face-colored maps}, i.e. maps whose faces carry colors among a set of $N$ colors (or equivalently
receive a weight $N^F$ if they have $F$ faces). {}From the matrix model analysis, the generating function of face-colored maps has indeed a simple 
expression in terms of the $r_i$'s for $i\leq N$. {}From our new interpretation of $r_i$, the generating functions for face-colored maps are shown to coincide with the generating functions for appropriately weighted marked maps. This allows us, by some appropriate counting of marked maps, to recover the celebrated Harer-Zagier formula~\cite{HaZa86} for face-colored rooted maps with a single vertex, as well as a special case of a formula by Goulden and Slofstra~\cite{GoSl10} for face-colored rooted maps with two vertices.

As opposed to the $R_i$'s, the $r_i$'s do not seem to have a simple closed expression. On the other hand, they satisfy remarkable and simple \emph{differential identities}
with no analog for the $R_i$'s. We show that these identities may receive a direct combinatorial explanation in terms of marked maps. 

\subsection{Plan of the paper}
The paper is organized as follows: Sections~\ref{sec:recursion} to \ref{sec:moreres} deal with Eulerian maps. We recall in Section~\ref{sec:recursion}
results from the matrix model enumeration approach: in Section~\ref{sec:intrep}, we discuss the representation of the generating function $r_1$ for
(rooted) Eulerian maps with unfixed genus in terms of a real integral (the $N=1$ version of $N\times N$ matrix integrals) and derive in \ref{sec:orthpol}
a recursive system of non-linear equations (Equation~\ref{eq:recurri}) for a family $(r_i)_{i\geq 1}$ of formal power series of which $r_1$ is the first term. Returning to $N\times N$ Hermitian
matrices, we then recall in Section~\ref{sec:facecol} how to obtain two alternative expressions for the generating function of face-colored Eulerian maps in terms of the $r_i$'s
only. The comparison between these expressions yields a set of differential identities for the $r_i$'s (Equation~\ref{drivalue}) which are proved in Appendix~\ref{sec:dri}
by algebraic manipulations on the recursive system for $r_i$. 

Section~\ref{sec:blossom} deals with blossoming trees, which are defined in Section~\ref{sec:blossomgen} as particular plane trees with two kinds of leaves, 
called opening and closing. We introduce more specifically \emph{Eulerian trees} as a particular subset of blossoming trees and derive a recursive system for their generating function in Section~\ref{sec:euleriantree}.  This allows us to interpret $r_1$
as the generating function for Eulerian trees with a weight $h$ per closing leaf of height $h$.  

Section~\ref{sec:bijtreemap} makes the connection between Eulerian trees and maps. We first discuss in Section~\ref{section:alpha} orientations of map edges with prescribed outdegrees of the vertices. We
recall in particular a bijection by Bernardi between so-called \emph{minimal orientations} and spanning trees on any given map. This property is
used in Section~\ref{sec:bij_eulerian} to design a bijection
between rooted Eulerian maps and Eulerian trees endowed with a forward matching (Theorem~\ref{thm:eul}), themselves in bijection with what we call \emph{enriched Eulerian trees}, which are trees whose closing leaves carry so-called matching indices and which are clearly enumerated by $r_1$. 
Section~\ref{sec:planar} presents a purely planar version of the bijection, where crossing-vertices are added to the blossoming trees,
leading to a canonical planar representation of maps with arbitrary genus. The number of crossing-vertices in the canonical planar
representation may then be controlled by changing the weight $h$ for leaves at height $h$ by its $q$-analog $[h]_q$.

Section~\ref{sec:moreres} explores an extension of our bijection by introducing \emph{marked Eulerian maps} (which are maps with particular marked oriented edges) in 
Section~\ref{sec:marked_maps}, in bijection with Eulerian trees endowed with a marked matching.  We then show in 
Section~\ref{sec:interpr_ri} how to interpret the $r_i$'s as generating functions for such marked Eulerian maps where the marked edges receive some multiplicities,
in bijection with particular enriched Eulerian trees (so called \emph{$i$-enriched}). {}From this new interpretation of $r_i$, we may obtain yet another, now purely combinatorial 
proof of the differential identities for the $r_i$'s (Equation~\ref{drivalue}), as detailed in Appendix~\ref{sec:driproof}.
Face-colored maps are discussed in Section~\ref{sec:face-colored}, where a striking identification 
of their generating function with that of some marked maps (Equation~\ref{eq:bijtT}) is derived. We also show there how to recover the Harer-Zagier formula for face-colored rooted maps with a single vertex.

We repeat all the above analysis in Section~\ref{sec:mregular}, now for the family of $m$-regular bipartite maps with unfixed genus.
We briefly recall in Section~\ref{sec:intmregular} results for their enumeration via matrix models, leading 
 again to a recursive system for a (new) family of functions $(r_i)_{i\geq 1}$ (of which $r_1$ is the desired rooted $m$-regular bipartite map generating function),
now coupled to a new family $(q_i)_{i\geq 1}$ (Equation~\ref{eq:sysrq}). A detailed proof of the recursive system via formal integrals is presented in Appendix \ref{sec:mregularint}. The $r_i$'s again satisfy remarkable differential identities (Equation~ \ref{eq:drimbis}), proved in Appendix~\ref{sec:drim} by algebraic manipulations on the recursive system.
We then describe the relevant blossoming trees in Section~\ref{sec:blossommregular} and show that rooted $m$-regular bipartite maps are
in bijection with particular bipartite blossoming trees, called \emph{$m$-bipartite trees}, endowed with a forward matching, 
or equivalently \emph{enriched $m$-bipartite trees}
(Theorem \ref{thm:bip}), which are counted by $r_1=1+q_1$. In Section~\ref{sec:morem}, we extend this bijection to a one-to-one correspondence between marked $m$-regular bipartite maps and marked $m$-bipartite trees, leading
to an interpretation of $q_i$ as the generating function for particular marked $m$-regular bipartite maps with multiplicities. {}From this interpretation, 
we can give in Appendix~\ref{sec:proofdqim} a combinatorial proof of the differential identities satisfied by $r_i$ (Equation~\ref{eq:drimbis}). 
We end our study with a discussion of face-colored $m$-regular bipartite maps, where we obtain again a remarkable identification 
of their generating function with that of properly weighted marked maps (Equation~\ref{eq:bijtTm}), and
reproduce a formula by Goulden and Slofstra for bipartite face-colored rooted maps with two vertices of degree $m$.

Section \ref{sec:conclusion} presents a discussion on a number of side results. We briefly propose in Section~\ref{sec:3regular} a possible extension of our
bijective construction to the case of maps with arbitrary (not necessarily even) vertex degrees, with an emphasis on the case of $3$-regular maps. 
We also discuss in Section~\ref{Contfrac} a strategy to obtain non-linear differential equations for the generating functions of maps of unfixed genus and bounded vertex degrees, and 
show the existence of simple continued fraction expressions for $r_1$ in particular cases of maps with small vertex degrees. 

\section{Expressing the generating function of Eulerian maps as the first term $r_1(t)$ in a recursive system}
\label{sec:recursion}
Recall that an Eulerian map is a map where all vertices have even degrees. In particular, this guarantees the existence of an 
Eulerian tour on the map. In this section, we recall some results on the enumeration of Eulerian maps, as obtained from the matrix integral approach. In particular, we show via some integral representation that the counting series $r_1(t)$ for rooted Eulerian maps with unfixed genus, 
enumerated with a weight $t$ per edge, is the first term of a family $(r_i(t))_{i\geq 1}$ of formal series in $t$ related by an infinite set of recursive equations
(Equation \ref{eq:recurri}).

\subsection{Integral representation of Eulerian map generating functions}
\label{sec:intrep}
The generating function for Eulerian maps enumerated with a weight $g_k$ per vertex of degree $2k$ ($k\geq 1$) and a weight $t$
per edge is given formally by
\begin{equation}
h_0(t)=\frac{1}{\sqrt{2\pi}}\int_{-\infty}^{+\infty}dx\, {\rm e}^{-\frac{x^2}{2}}\, {\rm e}^{V(t,x)}\ ,
\qquad 
V(t,x)=\sum_{k\geq 1} t^k\, g_k\, \frac{x^{2k}}{2k}\ .
\label{eq:h0}
\end{equation}
In \eqref{eq:h0}, the integrand ${\rm e}^{V(t,x)}$ is understood as a formal power series in the variable $t$, whose coefficients are polynomials
in the variable $x$ and, throughout the paper, for a formal power series 
\begin{equation*}
F(t,x):=\sum_{E\geq 0} t^E\, P_E(x)
\end{equation*}
with $P_E$ a sequence of polynomials, we denote
\begin{equation*}
\frac{1}{\sqrt{2\pi}} \int_{-\infty}^{+\infty}dx\, {\rm e}^{-\frac{x^2}{2}} F(t,x):=
\sum_{E\geq 0} t^E\, \frac{1}{\sqrt{2\pi}} \int_{-\infty}^{+\infty}dx\, {\rm e}^{-\frac{x^2}{2}} P_E(x)
\end{equation*}
which is also a formal power series in $t$. 

Now it is a classical result that, as a power series, we may write $h_0(t)=1+\sum_{E\geq 1}t^E Z_E$
where $Z_E$ denotes the generating function for possibly disconnected Eulerian maps with a total of $E$ edges, 
and enumerated with suitable symmetry factors. This is a consequence of the identity 
$\frac{1}{\sqrt{2\pi}} \int_{-\infty}^{+\infty}dx\, {\rm e}^{-\frac{x^2}{2}} x^{2s}=(2s-1)!!$ which is the number of pairings on a set
of $2s$ elements (here the set consists of the half edges that are to be paired to construct the map). 
To avoid symmetry factors, we may instead consider the generating functions 
$W_E$ for \emph{rooted\footnote{i.e. with a marked corner, or equivalently
a marked oriented edge.} connected Eulerian maps} with $E$ edges, enumerated with a weight $g_k$ per vertex of degree $2k$ ($k\geq 1$).
The corresponding power series $r_1(t)$ reads
\begin{equation}
r_1(t):=1+\sum_{E\geq 1}t^E W_E=1+2t\frac{d}{dt}{\rm Log}\,  h_0(t)
\label{eq:tdth0}
\end{equation}  
with a conventional first term $1$. {}From now on, all maps will be connected unless otherwise stated.
\subsection{Orthogonal polynomials and recursion relations}
\label{sec:orthpol}
As we shall now recall, the counting series $r_1(t)$ in \eqref{eq:tdth0} is the first term of an infinite family $(r_i(t))_{i\geq 1}$ of series
which are entirely determined by a recursive set of equations. This property may be established by introducing a family 
$p_i:=p_i(t,x)$ (with $i\in \mathbb{N}$) of orthogonal polynomials as follows. Defining, for two formal power series $F(t,x)$ and $G(t,x)$ in the variable $t$ whose 
coefficients are polynomials in $x$, the scalar product\footnote{More precisely, it is a symmetric bilinear form that returns
a power series in $t$. In practice, we shall only need the further property that $\langle F | F \rangle\neq 0$ if $F$ is non-zero.}
\begin{equation*}
\langle F | G \rangle:=\frac{1}{\sqrt{2\pi}}\int_{-\infty}^{+\infty}dx\, {\rm e}^{-\frac{x^2}{2}}\, {\rm e}^{V(t,x)}\, F(t,x)\, G(t,x)\ ,
\end{equation*}
the orthogonal polynomials $p_i(t,x)$ are defined by the conditions (which determine them entirely)
\begin{equation*}
\langle p_i | p_j \rangle =h_i(t) \, \delta_{i,j}\ ,  \qquad p_i(t,x)=x^i+\sum_{k<i} a_{i,k}(t) x^k
\end{equation*}
(note that $h_0(t)$ matches its definition given by \eqref{eq:h0} since $p_0(t,x)=1$). Now clearly, since $V(t,x)$ is an even function
of $x$,  the polynomials $(-1)^ip_i(t,-x)$ satisfy the desired conditions and, by unicity, $p_i$ is even in $x$ for even $i$ and odd for odd $i$.
In particular, we deduce that $p_1(t,x)=x$.
We may then write
\begin{equation*}
x\, p_i(t,x)=p_{i+1}(t,x)+r_i(t)\, p_{i-1}(t,x)+\sum_{k\geq 2} \alpha_{i+1-2k}(t) p_{i+1-2k}(t,x)\ , \quad i\geq 1
\end{equation*}  
where the sum runs over non-negative values of $i+1-2k$ with $k\geq 2$ (we will see below that the coefficient $r_1(t)$ in
this equation is precisely the counting function defined in \eqref{eq:tdth0}). For $i=0$, we have the simplified relation
$x\, p_0(t,x)=p_1(t,x)$.
{}From the identity $\langle x\, p_i | p_h\rangle= \langle p_i |x p_h\rangle =0 $ for $i> h+1$ (since $xp_h$ is a linear combination or $p_m$'s with $m\leq h+1$), i.e for $h <i-1$, we deduce that
all the $\alpha_{i+1-2k}$ are $0$, so that we may eventually write
\begin{equation}
x\, p_i(t,x)=p_{i+1}(t,x)+r_i(t)\, p_{i-1}(t,x)\ .
\label{eq:xpi}
\end{equation}  
{}From the identity  $\langle x\, p_i | p_{i-1}\rangle= \langle p_i |x p_{i-1}\rangle$, we then deduce that 
$h_{i-1}(t) r_i(t)= h_i(t)$, hence
\begin{equation*}
r_i(t)=\frac{h_i(t)}{h_{i-1}(t)}\ , \qquad i\geq 1\ .
\end{equation*} 
As for the power series \eqref{eq:tdth0}, we have
\begingroup
\allowdisplaybreaks
\begin{align*}
2t\frac{d}{dt}{\rm Log}\,  h_0(t)&= \frac{1}{h_0(t)}\frac{1}{\sqrt{2\pi}}\int_{-\infty}^{+\infty}dx\, {\rm e}^{-\frac{x^2}{2}}\, 2t\frac{\partial}{\partial t}{\rm e}^{V(t,x)}\\
&= \frac{1}{h_0(t)}\frac{1}{\sqrt{2\pi}}\int_{-\infty}^{+\infty}dx\, {\rm e}^{-\frac{x^2}{2}}\, x\frac{\partial}{\partial x}{\rm e}^{V(t,x)}\\
&= -\frac{1}{h_0(t)}\frac{1}{\sqrt{2\pi}}\int_{-\infty}^{+\infty}dx\, \frac{d}{dx}\left(x\, {\rm e}^{-\frac{x^2}{2}}\right)\, {\rm e}^{V(t,x)}\\
&= -\frac{1}{h_0(t)}\frac{1}{\sqrt{2\pi}}\int_{-\infty}^{+\infty}dx\, {\rm e}^{-\frac{x^2}{2}}\, {\rm e}^{V(t,x)}+\frac{1}{h_0(t)}\frac{1}{\sqrt{2\pi}}\int_{-\infty}^{+\infty}dx\, x^2\, {\rm e}^{-\frac{x^2}{2}}\, {\rm e}^{V(t,x)}\\
&=-1+\frac{h_1(t)}{h_0(t)}
\end{align*}
\endgroup
where we used $x=p_1(t,x)$. This leads to the desired relation \eqref{eq:tdth0} which identifies the coefficient $r_1(t)$ in \eqref{eq:xpi}
as the generating function for rooted Eulerian maps, counted  by their number of edges.

The formal power series $r_i(t)$ satisfy an infinite system of recursion relations which determines them all order by order in $t$ and which may be obtained as follows:
we start with the identity
\begin{equation*}
\frac{\partial}{\partial x}w(t,x)+\Big(x-\sum_{k\geq 1} t^k g_{k}x^{2k-1}\Big)\, w(t,x)=0\ , \quad w(t,x):= {\rm e}^{-\frac{x^2}{2}+V(t,x)} 
\end{equation*}
which allows to write
\begingroup
\allowdisplaybreaks
\begin{align*}
\!\!\!\!\!\!\!\!\!\!\!\!\!\!\! h_{i}(t)&=\langle p_i |x\, p_{i-1}\rangle\\
&=\frac{1}{\sqrt{2\pi}}\int_{-\infty}^{+\infty}dx\, w(t,x)\, p_i(t,x)\, x\, p_{i-1}(t,x)\\
&=\frac{1}{\sqrt{2\pi}}\int_{-\infty}^{+\infty}dx\, \Bigg(-\frac{\partial}{\partial x}w(t,x)+\Big(\sum_{k\geq 1} t^k g_{k}x^{2k-1}\Big)\, w(t,x) \Bigg) p_i(t,x)\, p_{i-1}(t,x)\\
&=\frac{1}{\sqrt{2\pi}}\int_{-\infty}^{+\infty}dx\, w(t,x)\, \Bigg(\frac{\partial}{\partial x}p_i(t,x)\, p_{i-1}(t,x)+p_i(t,x)\, \frac{\partial}{\partial x}p_{i-1}(t,x)\\
& \qquad  \qquad \qquad \qquad \qquad \qquad \qquad \qquad \qquad+\Big(\sum_{k\geq 1} t^k g_{k}x^{2k-1}\Big)\, p_i(t,x)\, p_{i-1}(t,x)\Bigg)\\
&=\langle \frac{\partial}{\partial x}p_{i}| p_{i-1}\rangle+\underbrace{\langle p_i | \frac{\partial}{\partial x}p_{i-1}\rangle}_{=0}+
\sum_{k\geq 1} t^k g_{k} \langle p_{i-1} | x^{2k-1} p_i\rangle\\
&=i\, h_{i-1}(t)+\sum_{k\geq 1} t^k g_{k} \langle p_{i-1} | x^{2k-1} p_i\rangle\\
\end{align*}
\endgroup
where the last identity follows from the relation $\frac{\partial}{\partial x}p_i(t,x)=i\, p_{i-1}(t,x)+\sum_{k<i} k\, a_{i,k}(t) x^{k-1}- i \sum_{k<i-1} a_{i-1,k}(t) x^k$.
We end up with the relation
\begin{equation*}
r_i(t)=i+\sum_{k\geq 1} t^k g_{k} \frac{1}{h_{i-1}(t)}\langle p_{i-1} | x^{2k-1} p_i\rangle\ .
\end{equation*}
Using repeatedly the relation $x\, p_h(t,x)=p_{h+1}(t,x)+r_h(t)\, p_{h-1}(t,x)$ for $h\geq 1$ and $x\, p_0(t,x)=p_1(t,x)$, and following the variation of the index $h$, the quantity 
$\frac{1}{h_{i-1}(t)}\langle p_{i-1} | x^{2k-1} p_i\rangle$ may be interpreted as the weighted sum over the set $\mathcal{P}_k^{(i)}$ of Dyck paths
$\wp$ of length $2k-1$ (whose height is
the running index $h$ when repeating the relation) from height $i$ to height $i-1$
where each path is enumerated with a weight $r_h(t)$ for each descending step $h\to h-1$. We deduce the recursion relations
\begin{equation}
r_i(t)=i+\sum_{k\geq 1} t^k g_{k} \sum_{\wp \in \mathcal{P}_k^{(i)}} 
\prod_{\hbox{\tiny{descending steps}}\atop h\to h-1\ \hbox{\tiny{of}}\ \wp} r_h(t)\ , \quad i\geq 1\ .
\label{eq:recurri}
\end{equation}
To summarize, the generating function $r_1(t)=1+\sum_{E\geq 1}t^E W_E$ for rooted connected Eulerian maps is obtained as the first term $i=1$
in the above recursive system defining the $r_i$'s and may be obtained order by order in $t$ from this system.

As a simple example, let us consider the case of $4$-regular maps, i.e. take\footnote{For $4$-regular maps, the number of vertices is half the number of edges
so we may decide to take $g_2=1$.} $g_k=\delta_{k,2}$. The above recursive system
then reduces to 
 \begin{equation*}
r_i(t)=i+t^2\, r_i(t)\big(r_{i+1}(t)+r_i(t)+r_{i-1}(t)\big)\ , \qquad i\geq 1
\end{equation*}
with the convention $r_0(t)=0$. At first orders in $t$, this yields
\begin{equation*}
r_i(t)=i+3i^2\, t^2+6i(3i^2+1)t^4+27i^2(5i^2+6)t^6+18i(63i^4+174i^2+35)t^8+\ldots
\end{equation*}
and in particular, from the $i=1$ series, 
\begin{equation*}
W_2= 3\ , \quad 
W_4= 24 \ , \quad 
W_6= 297 \ , \quad 
W_8= 4896\ . \quad 
\end{equation*}
\subsection{Generating functions of face-colored maps}
\label{sec:facecol}
The above results correspond to the $N=1$ version of a more general framework involving integrals over $N\times N$ Hermitian matrices. The effect of replacing
the real integrals above by matrix integrals is then to give weight $N$ to each face of the map 
(this weight comes from the summation over matrix element indices). If we denote by $Z_{E,F}$ the generating function for (not necessarily connected) Eulerian maps with $E$ edges and $F$ faces, 
enumerated with a weight $g_k$ per vertex of degree $2k$ (and appropriate symmetry factors), we now get from the Wick formula the integral 
expression\footnote{In the matrix integral approach, it is customary
to also give a weight $N$ to the $V$ vertices of the map and $N^{-1}$ to its $E$ edges so that the map gets a total weight $N^{V-E+F}=N^{2-2h}$ if
it has genus $h$, see~\cite{PDF06,DGZ95}. This is done by modifying the term $ \exp\left(-{\rm Tr}\left(H^2/2+V(t,H)\right)\right)$ in the matrix integral into 
$ \exp\left(-N{\rm Tr}\left(H^2/2+V(t,H)\right)\right)$ (and adapting the normalization prefactor).
We will not discuss this alternative weighting here.} 
\cite[Chap.3]{LZ04}:
\begin{equation*}
H_0(t,N):=\sum_{E,F\geq 1}t^E N^F\,  Z_{E,F}=\frac{2^{N(N-1)/2}}{(2\pi)^{N^2/2}}\int dH\, {\rm e}^{-{\rm Tr}\left(\frac{H^2}{2}+V(t,H)\right)}\ ,
\end{equation*}
where the integral is over $N\times N$ Hermitian matrices with $dH$ the Lebesgue measure $dH=\prod\limits_{i} dH_{i,i} \prod\limits_{i<j}
d{\rm Re}(H_{i,j})d{\rm Im}(H_{i,j})$ and $V$ as in \eqref{eq:h0}.
The matrix integral may then be computed by use of the orthogonal polynomials
$p_i(t,x)$ of the previous section,
with the result \cite[Sect.3.5]{LZ04}: 
\begin{equation}
H_0(t,N)=N!\prod_{i=1}^{N}h_{i-1}(t)=N! \prod_{i=1}^{N}\big(r_i(t)\big)^{N-i}\times \big(h_0(t)\big)^N
\label{eq:Zval}
\end{equation}
with $h_i(t)$ and $r_i(t)$ as in Section~\ref{sec:orthpol}.
As before, we may consider instead the associated generating functions $W_{E,F}$ for rooted connected Eulerian maps. We then have
\begin{equation}
\begin{split}
T(t,N)&:=\sum_{E,F\geq 1}t^E N^F\,  W_{E,F}\\
&=2t\frac{d}{dt}{\rm Log}\,  H_0(t,N) \\
&=N\times 2t\frac{d}{dt}{\rm Log}\,  h_0(t)+\sum_{i=1}^N(N-i)\, 2t\frac{d}{dt}{\rm Log}\,  r_i(t)\\
&=N \big(r_1(t)-1\big)+\sum_{i=1}^N(N-i)\, 2t\frac{d}{dt}{\rm Log}\,  r_i(t)
\end{split}
\label{eq:Tval}
\end{equation}
where we used \eqref{eq:tdth0} in the last line.
Since $N$ is necessarily an integer in the above formula, the quantity $T(t,N)$ may be interpreted as the counting series of \emph{face-colored} rooted
Eulerian maps, i.e. maps whose faces are colored with color set $\{1,2,\ldots,N\}$.

Another simpler expression for $T(t,N)$ can be obtained by identifying rooted Eulerian maps with Eulerian maps with a marked vertex of degree $2$
and its incident half-edges distinguished (sometimes called two-leg maps, the identification simply amounts to put a vertex of degree $2$ in the middle of the root edge). This yields 
\begin{equation*}
T(t,N)=\frac{1}{H_0(t,N)}\frac{2^{N(N-1)/2}}{(2\pi)^{N^2/2}} \int dH\, {\rm Tr}(H^2) \, {\rm e}^{-{\rm Tr}\left(\frac{H^2}{2}+V(t,H)\right)}-N^2\ ,
\end{equation*}
where ${\rm Tr}(H^2)$ accounts for the marked vertex of degree $2$ and the subtracted term removes the contribution $N^2$ of the trivial map
made of a single vertex of degree $2$ with an incident loop.  
In the orthogonal polynomial 
formalism, this yields the alternative expression 
 \begin{equation}
 \begin{split}
T(t,N)&= \sum_{i=1}^N\frac{\langle p_{i-1} | x^2\, p_{i-1}\rangle}{h_{i-1}(t)}-N^2\\
&=\sum_{i=1}^N \big(r_{i}(t)+r_{i-1}(t)\big)-N^2\\
&=S(t,N)+S(t,N-1)\ , \qquad S(t,N):=\sum_{i=1}^N \big(r_i(t)-i\big) \ .
\label{eq:Tvalbis}
\end{split}
\end{equation} 
In particular, the identification\footnote{More precisely, we consider for $N\geq 1$ the quantity $(T(t,N+1)-T(t,N))-(T(t,N)-T(t,N-1))=T(t,N+1)-2T(t,N)+T(t,N-1)$, with the convention $T(t,0)=0$. By~\eqref{eq:Tval} it equals  $2t\frac{d}{dt}{\rm Log}\, r_N(t)$, and by~\eqref{eq:Tvalbis} it equals $r_{N+1}(t)-r_{N-1}(t)-2$.} of the two expressions \eqref{eq:Tval} and \eqref{eq:Tvalbis} for $T(t,N)$ for some arbitrary $N$  
implies the following
remarkable identity satisfied by the functions $r_i(t)$:
\begin{equation}
2t\frac{d}{dt}{\rm Log}\, r_i(t)=r_{i+1}(t)-r_{i-1}(t)-2\ , \quad i\geq 1
\label{drivalue}
\end{equation}
(with $r_0(t)=0$ as before). 
This identity is proved in Appendix \ref{sec:dri} by verification from the recursive system \eqref{eq:recurri} itself
(without recourse to face-colored maps). We also present another, purely combinatorial proof of \eqref{drivalue} in Appendix \ref{sec:driproof} 
in terms of so-called marked maps and blossoming trees.
Note finally that for $N=1$, the identity \eqref{eq:Tvalbis} yields directly $T(t,1)=r_1(t)-1$, which provides yet another proof that $r_1(t)$ is  the generating function for
(uncolored) rooted Eulerian maps. 

\section{Interpreting $r_1(t)$ as a counting series for some blossoming trees}
\label{sec:blossom}
In this section, we show that, as solutions of the recursive system \eqref{eq:recurri}, the $r_i$'s have a natural interpretation as 
counting series for properly weighted ``Eulerian trees'', defined below as a particular family of blossoming trees. This holds in particular for $r_1$ itself, 
which corresponds to the counting of so-called ``balanced'' Eulerian trees with appropriate leaf weights.
\subsection{Blossoming trees: generalities}
\label{sec:blossomgen}
All trees considered here are plane trees. The degree of a vertex is its number of neighbours. The leaves are thus the vertices of degree $1$, the other vertices are called \emph{nodes}. 
A \emph{rooted tree} is a tree with a marked leaf. It is called \emph{nodeless} if it has just one edge (connecting two leaves),
otherwise the node adjacent to the root leaf is called the \emph{root node}.  A \emph{blossoming tree} $T$ is a rooted tree with two kinds of leaves: \emph{opening leaves} and \emph{closing leaves}, such that there are as many opening as closing leaves, and 
 the root leaf is opening. The \emph{leaf-path} $w(T)$ of $T$ is the path, starting at height $1$, obtained from a clockwise walk around the tree (with the outer face on the left) starting and ending at (but not including) the root leaf, where an up-step (resp. down-step) is drawn when passing along an opening (resp. closing) leaf. If $T$ has $2s$ leaves, then $w(T)$ has length $2s-1$ and ends at height $0$. A closing leaf is said to have \emph{height} $h$ 
if the corresponding down-step in $w(T)$ descends from $h$ to $h-1$. The tree $T$ is called \emph{balanced} if the height of every leaf is positive\footnote{Otherwise
stated, in a clockwise walk around the tree starting at (and including) the root leaf, the number of encountered opening leaves is always larger or equal to that of
closing leaves.},
so that $w(T)$ is a Dyck path\footnote{Here and throughout the paper, a Dyck path is a walk on $\mathbb{N}$ with elementary steps in $\{+1,-1\}$ and with prescribed 
starting and ending points.} of length $2s-1$ from height $1$ to height $0$. 
More generally, for $i\geq 1$, we let $w^{(i)}(T)$ be the vertical shift of $w(T)$ that starts at height~$i$.
The tree $T$ is called \emph{$i$-balanced} if $w^{(i)}(T)$ is in $\cP_s^{(i)}$, the set of Dyck paths of length $2s-1$ from height $i$ to height $i-1$. A closing leaf is said to have \emph{$i$-height} $h$
if the corresponding down-step in $w^{(i)}(T)$ descends from height $h$ to $h-1$ (with $h\geq 1$ if $T$ is $i$-balanced). Note that the $i$-height of a closing leaf 
is nothing but $(i-1)$ plus its height, and that a $1$-balanced tree is nothing but a balanced tree.

An \emph{Eulerian tree} is a blossoming tree such that each node has even degree, and every node $v$ of degree $2k$
has $k-1$ adjacent opening leaves (not counting the root leaf if $v$ is the root node). It is easy to check that an Eulerian tree satisfies the required condition that the number of opening and closing leaves are equal (indeed, if $n_k$ denotes the number of nodes of degree $2k$, the number of edges is $1+\sum_k (2k-1)n_k$ hence the number of vertices is $2+\sum_k (2k-1)n_k$; since the tree has $\sum_k n_k$ nodes it has a total of $2+2\sum_k (k-1)n_k$ leaves, and clearly by definition the number of opening leaves is $1+\sum_k (k-1)n_k$).

\subsection{Counting series of Eulerian trees}\label{sec:euleriantree}
\label{sec:blossomcount}
For $i\geq 1$, we let $\hr_i(t)$ be the counting series (in the variable $t$ and weight-parameters $(g_k)_{k\geq 1}$ and $(z_h)_{h\geq 1}$) of $i$-balanced Eulerian trees, where each node of degree $2k$
is weighted by $t^kg_k$, and each closing leaf of $i$-height $h$ is weighted by $z_h$ (the parameter $t$ is conjugate to the total half-degree of the nodes).  We claim that the series
$\hr_1,\hr_2,\cdots$ satisfy the recursion relations
\begin{equation}
\hr_i(t)=z_i+\sum_{k\geq 1} t^k g_{k} \sum_{\wp \in \mathcal{P}_k^{(i)}} 
\prod_{\hbox{\tiny{descending steps}}\atop h\to h-1\ \hbox{\tiny{of}}\ \wp} \hr_h(t)\ , \quad i\geq 1\ .
\label{eq:recurhri}
\end{equation}
Clearly the nodeless Eulerian tree is represented by the term $z_i$. 
For $T$ an $i$-balanced Eulerian tree with a root node $v$ of degree $2k$, let $e_0,e_1,\ldots,e_{2k-1}$ be the 
sequence of incident edges in clockwise order around $v$, with $e_0$ the one leading to the root leaf. 
Then the \emph{root path} of $T$ is the path $\wp$ of length $2k-1$, starting at height $1$ and ending at height $0$,   
 such that for $j\in[1..2k-1]$ the $j$th step of $\wp$
 is an up-step if $e_j$ leads to an opening leaf, and is a down-step otherwise. For $i\geq 1$ we let $\wp^{(i)}$ be the vertical shift of $\wp$ that starts at height $i$ (and ends at height $i-1$).  
Let $j_1<\cdots<j_k$ be the indices
of the down-steps of $\wp$, and let $T_{1},\ldots,T_{k}$ be the subtrees attached at each of the edges $e_{j_1},\ldots,e_{j_k}$. Then clearly $w^{(i)}(T)$ is obtained from $\wp^{(i)}$ by replacing, for each $r\in[1..k]$, the down-step $d_r$ at position $j_r$ by the path $w^{(h_r)}(T_{r})$, with $h_r$ the starting height of $d_r$, see Figure~\ref{fig:hri} for an illustration. In this substitution, note that $T$ is $i$-balanced iff $\wp^{(i)}\in \cP_k^{(i)}$ and
$T_r$ is $h_r$-balanced for each $r$ (note that the case where $T_r$ reduces to the nodeless  Eulerian tree, such as $T_2$ in Figure~\ref{fig:hri}, corresponds to the situation where $e_{j_r}$ leads to a closing leaf of $i$-height $h_r$). This readily yields~\eqref{eq:recurhri}. 

The series $r_i(t)$, as defined in~\eqref{eq:recurri}, are a specialization of $\hr_i(t)$, obtained by setting $z_h=h$. Thus $r_i(t)$ can be interpreted as the counting series of $i$-balanced Eulerian trees, where each closing leaf of $i$-height $h$ receives a weight $h$, or equivalently is decorated by 
an integer index $\iota\in [0..h-1]$.  

\begin{figure}
\begin{center}
\includegraphics[width=14cm]{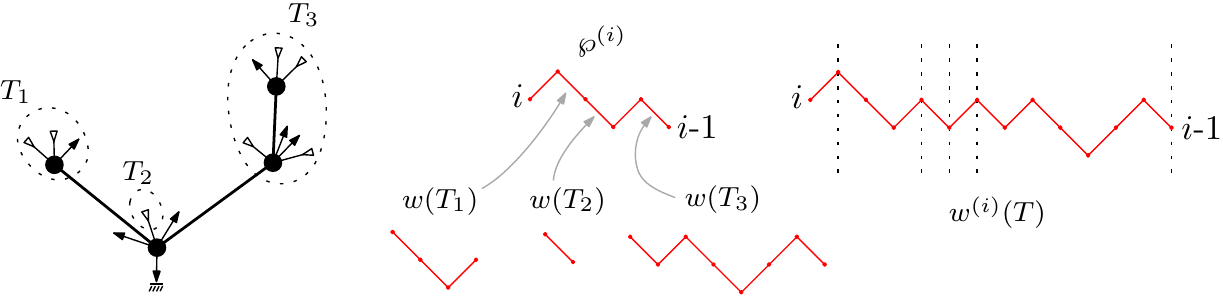}
\end{center}
\caption{An Eulerian tree $T$ (left side, opening leaves are drawn as outgoing black arrows, closing leaves as ingoing white arrows). The shifted leaf-path $w^{(i)}(T)$ is obtained from the shifted root path $\wp^{(i)}$ of $T$ where the descending steps are replaced by the (properly shifted) leaf-paths of the subtrees attached to the root node.}
\label{fig:hri}
\end{figure}

\section{Bijection between some blossoming trees and rooted Eulerian maps}
\label{sec:bijtreemap}
This section presents a bijection between rooted Eulerian maps and so-called ``enriched'' Eulerian trees or, equivalently, Eulerian trees
endowed with a so-called ``forward matching'' of their leaves (Theorem \ref{thm:eul}). This gives a combinatorial explanation of why the solution $r_1(t)$ of the system 
\eqref{eq:recurri}, which clearly enumerates the 
above decorated Eulerian trees, is also the generating function for rooted Eulerian maps.
\subsection{Minimal orientations with prescribed outdegrees}\label{section:alpha}
For a rooted map $M$ (map with a marked corner) of arbitrary genus, the \emph{root vertex} $v_0$ is the one incident to the root corner, and the \emph{root half-edge} is the half-edge just after the root corner in clockwise order around $v_0$. The \emph{root edge} is the edge containing the root half-edge. Let $V$ and $E$ be the vertex-set and edge-set of $M$. For $\alpha:V\to\mathbb{N}$, an \emph{$\alpha$-orientation} of $M$ is an orientation of the edges of $M$ such that every vertex $v$ has outdegree $\alpha(v)$. The function~$\alpha$ is called \emph{feasible} if $M$ admits an $\alpha$-orientation. For $v'$ a vertex of $M$, an orientation of $M$ is called \emph{$v'$-accessible} if for every vertex $v$ there exists an oriented path from $v$ to $v'$. It is known that, for a given feasible~$\alpha$, either all $\alpha$-orientations are $v'$-accessible or none. In the first case, the  feasible function~$\alpha$ is called \emph{$v'$-accessible}.
 For $S\subseteq V$ let $\alpha(S):=\sum_{v\in S}\alpha(v)$ and let $E_S$ be the set of edges with both ends in $S$. It is known (see e.g.~\cite[Sect.2.1]{Fe03} and~\cite[Lem.3]{BeFu12}) that a function $\alpha$ is feasible and $v'$-accessible iff 
$\alpha(V)=|E|$ and for every $S\subseteq V\backslash v'$ one has $\alpha(S)>|E_S|$ (heuristically, it means that there exists at least one oriented edge
that allows to go away from $S$). A feasible function $\alpha$ is called \emph{root-accessible} if it is $v_0$-accessible.

\begin{figure}
\begin{center}
\includegraphics[width=14cm]{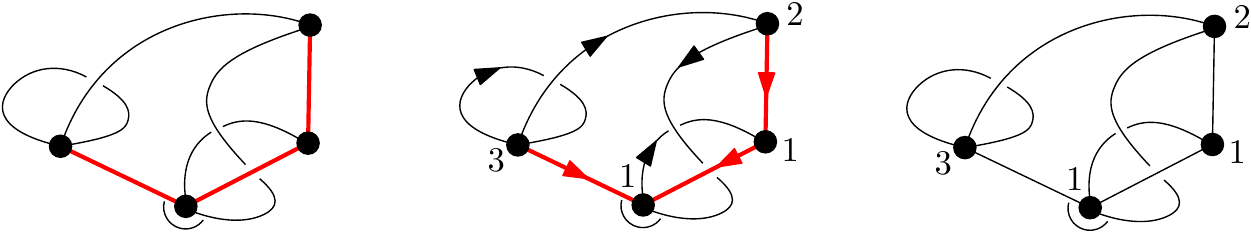}
\end{center}
\caption{Bernardi's bijection $\Gamma_M$ (for $M$ a rooted map): from a spanning tree $T$ of $M$ (left drawing, tree-edges are shown red) to a root-accessible feasible function $\alpha$ for $M$ (right drawing). The middle drawing shows the orientation $\Phi_M(T)$ and the vertex-outdegrees.}
\label{fig:bij_bernardi}
\end{figure}  

In \cite{Be08}, Bernardi gives a nice bijection $\Gamma_M$ between the spanning trees of $M$ and the feasible root-accessible functions $\alpha$ for $M$. 
For $T$ a spanning tree of $M$, the edges of $T$ are called \emph{internal} and the edges of $M\backslash T$ are called
\emph{external}. The half-edges of external edges are ordered according to a clockwise walk around $T$ starting at the root corner. Thus every external edge has a first half-edge and a second half-edge (the first one appearing before the second according to the above ordering). Let $\Phi_M(T)$ be the orientation of $M$ where internal edges are directed toward the root along $T$, and every external edge has its first half-edge outgoing and second half-edge ingoing: clearly $\Phi_M(T)$
is root-accessible (following the oriented internal edges from any vertex leads to the root). The mapping $\Gamma_M$ associates to~$T$ the outdegree sequence of $\Phi_M(T)$, see Figure~\ref{fig:bij_bernardi}. The fact that $\Gamma_M$ 
is bijective is equivalent to the following statement.

\begin{theo}[Bernardi (item 5 in Theorem 41 of~\cite{Be08})]
\label{thm:bernardi}
Let $M$ be a rooted map with vertex-set~$V$. For every feasible root-accessible function $\alpha:V\to\mathbb{N}$ there is a unique spanning tree $T$ of $M$ such that $\Phi_M(T)$ is an $\alpha$-orientation. That orientation is called 
 the \emph{minimal} $\alpha$-orientation of $M$.  
\end{theo}

Let us comment on how the minimal $\alpha$-orientation is computed, since it is a key ingredient in our bijections from maps to blossoming trees. 
For $M$ a rooted planar map, the minimal $\alpha$-orientation $O_{\rm min}$ is the unique $\alpha$-orientation of $M$ with no counterclockwise cycle~\cite{PS06,Be07}, and the spanning tree $T$ is computed from a certain traversal procedure applied to $O_{\rm min}$. For $M$ a rooted map of  arbitrary genus, as explained in~\cite{Be08}, the minimal $\alpha$-orientation $O_{\rm min}$ and the corresponding spanning tree $T$ of $M$  are computed jointly starting from a given $\alpha$-orientation $O$, using an adapted traversal procedure and cycle-reversal operations. Precisely, for $h$ an half-edge, ${\rm opp}(h)$ denotes the opposite half-edge on the same edge and $\sigma(h)$ denotes the next half-edge after $h$ in clockwise order around the incident vertex. The traversal of $M=(V,E)$ consists of $2|E|$ steps, where at each step $k\in[0..2|E|-1]$ a new half-edge $h_k$ is considered, starting with $h_0$ the root half-edge. The operations when considering $h_k$ are as follows, where a half-edge is called \emph{visited} if it is in $\{h_0,\ldots,h_{k-1}\}$ and an edge is called \emph{visited} if at least one of its two half-edges is visited.
\begin{itemize}
\item 
If $h_k$ is outgoing and ${\rm opp}(h_k)$ is unvisited, we move to $h_{k+1}:=\sigma(h_k)$.
\item
If $h_k$ is outgoing and ${\rm opp}(h_k)$ is visited, we move to $h_{k+1}:=\sigma({\rm opp}(h_k))$.
\item 
If $h_k$ is ingoing and ${\rm opp}(h_k)$ is visited, we move to $h_{k+1}:=\sigma(h_k)$. 
\item
If $h_k$ is ingoing and ${\rm opp}(h_k)$ is unvisited, then there are two cases, with $e$ the edge containing $h_k$: if there is a directed cycle $C$ of unvisited edges passing by $e$ then we reverse the orientations of the edges on $C$ and move to $h_{k+1}:=\sigma(h_k)$,  otherwise we move to $h_{k+1}:=\sigma({\rm opp}(h_k))$ and declare the edge $e$ as an internal edge.
\end{itemize}
 The procedure outputs $T$ as the set of internal edges, and $O_{\rm min}$ as the  orientation obtained after the last step (the order $h_0,\ldots,h_{2|E|-1}$ of the half-edges corresponds to a clockwise walk around $T$ starting at the root corner). 
Clearly the complexity of each step is of order at most $|E|$ so that the overall time complexity is of order at most $|E|^2$. 

\medskip
We will need the following lemma in the next section (for $k=1$), and also later on.

\begin{figure}
\begin{center}
\includegraphics[width=4cm]{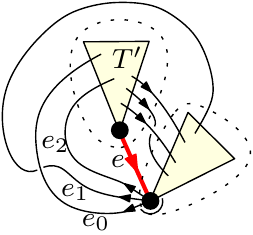}
\end{center}
\caption{The situation in the proof of Lemma~\ref{lem:outroot}.}
\label{fig:min_lemma}
\end{figure}

\begin{lem}\label{lem:outroot}
Let $M$ be a rooted map and $\alpha$ a feasible root-accessible function for $M$. Let $v_0$ be the root vertex, $d$ its degree,  and $h_0,\ldots,h_{d-1}$ the incident half-edges in clockwise order around $v_0$, with $h_0$ the root half-edge. 
For $k\in[1..d]$, assume there is an $\alpha$-orientation $O$ of $M$ such that $h_0,\ldots,h_{k-1}$ are outgoing. Then $h_0,\ldots,h_{k-1}$ are also outgoing in the minimal $\alpha$-orientation $O_{\rm min}$ of $M$.  
\end{lem}
\begin{proof}
Assume the statement does not hold, and let $j<k$ be such that (in the orientation $O_{\rm min}$) $h_0,\ldots,h_{j-1}$ are outgoing and $h_j$ is ingoing at $v_0$. Then necessarily $h_0,\ldots,h_{j-1}$ are parts of external edges denoted $e_0,\ldots,e_{j-1}$, while $h_j$ is the ingoing part of an internal edge~$e$ (it cannot be the 
ingoing part of an external edge since the other half-edge would be $h_m$ for some $m<j$ and this would make it impossible to have $h_m$ and $h_j$ both
outgoing in $O$) . 
Let $T'$ be the component of $T\backslash e$ that contains the origin of $e$, and let $S'$ be the set of vertices that are in $T'$.    
The \emph{cut} for a subset $S\subset V$ is the set of edges between $S$ and $V\backslash S$, and the \emph{demand} of $S$
is the number of edges of the cut that go from $V\backslash S$ to $S$. The demand of $S$ is the same
in every $\alpha$-orientation as it is equal to $\alpha(V\backslash S)-|E_{V\backslash S}|$. 
Then, as illustrated in Figure~\ref{fig:min_lemma}, the edges in the cut for $S'$ and different from $e_0,\ldots,e_{j-1}$ all go from $S'$ to $V\backslash S'$. 
Hence, if we let $F$ be the set of edges among $e_0,\ldots,e_{j-1}$ that are in the cut for $S'$, then the demand of $S'$ is $|F|$. Hence in every $\alpha$-orientation such that $e_0,\ldots,e_{j-1}$ are outgoing at $v_0$ (such as $O$ itself), the $|F|$ edges that contribute to the demand of $S'$ are those of $F$, and thus $e$ has to be ingoing at $v_0$ for the orientation $O$, yielding a contradiction.   
\end{proof}

\begin{figure}
\begin{center}
\includegraphics[width=14cm]{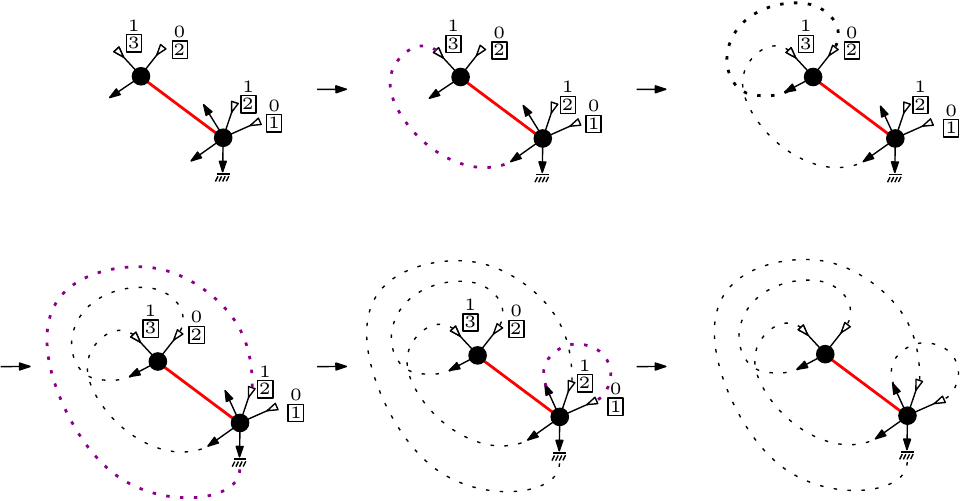}
\end{center}
\caption{{}From a balanced blossoming tree (here Eulerian) endowed with a matching assignment (for each closing leaf, its height is shown framed, above it is the matching-index) to the same blossoming tree endowed with a forward matching. 
}
\label{fig:bij_matching}
\end{figure}

\subsection{Application to Eulerian maps}\label{sec:bij_eulerian}
For $T$ a balanced blossoming tree, a \emph{matching-assignment} of $T$ is the assignment to each closing leaf $c$ of height $h$ of an integer $\iota(c)\in[0..h-1]$, which is called the \emph{matching-index} of $c$. An \emph{enriched blossoming tree} is a  balanced blossoming tree endowed with a matching-assignment.     
As we have seen in Section~\ref{sec:euleriantree}, $r_1(t)$ is the counting series of enriched Eulerian trees.   
 As an application of the previous section we are going to prove that such trees are in bijection with rooted Eulerian maps. 
Recall that the leaves of a blossoming tree $T$  are ordered according to a clockwise walk around the tree, starting with the root leaf. A \emph{forward matching} of $T$ is a matching of opening leaves with closing leaves, such that for each pair the opening leaf appears before the closing leaf. Note that then the tree $T$ is necessarily balanced.  As a first step, we argue that for $T$ a balanced blossoming tree, the matching assignments of $T$ may be identified\footnote{This is an adaptation to our setting of a well known construction to encode a matching by a decorated Dyck path, see~\cite{Fl80,MeVi94} and references therein.} with the forward matchings of $T$. 
Indeed, from a matching-assignment of $T$, we may construct a forward matching of $T$ step by step by treating the closing leaves
in their order of appearance in a clockwise walk starting from the root, see Figure~\ref{fig:bij_matching}. Every time we visit a closing leaf $c$, the height $h$ of $c$ corresponds to the number of opening leaves $o_1,\ldots,o_h$ that appear before $c$ and are not yet matched; if $c$ has matching-index $\iota$ then it is matched with $o_{h-\iota}$.

\begin{figure}
\begin{center}
\includegraphics[width=14.5cm]{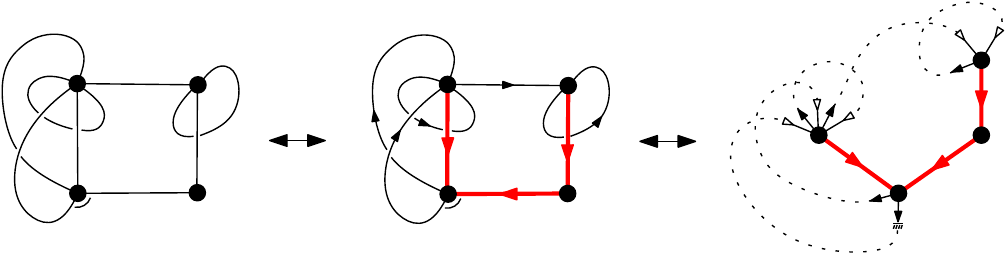}
\end{center}
\caption{Left: a rooted Eulerian map $M$. Middle: the minimal Eulerian orientation of $M$ (with the associated spanning tree in red). Right: cutting each external edge in its middle and taking as the root leaf the one at the end of the root half-edge, one obtains an Eulerian tree endowed with a forward matching.}
\label{fig:bij_eulerian}
\end{figure}

\bigskip
Now we show that Eulerian trees endowed with a forward matching are in bijection with rooted Eulerian maps, a consequence of the results of the previous section. For $M$ a rooted Eulerian map with vertex set $V$, we let
 $\alpha$ be the function that assigns to every vertex its half degree. For this function $\alpha$, an $\alpha$-orientation is called \emph{Eulerian}. Clearly the function $\alpha$ is feasible and
 root-accessible: this can be checked by the general existence criterion, or from the existence of an Eulerian tour. Note also that no edge is rigid (i.e., with the same direction in all Eulerian orientations), again due to the existence of an Eulerian tour, which can be reversed. In particular, the root half-edge can be
 chosen as outgoing. Let $O$ be the minimal Eulerian orientation of $M$, with $T$ the associated spanning tree (i.e., such that $\Phi_M(T)=O$).  By Lemma~\ref{lem:outroot}, the root half-edge is going out of the root vertex. We may then cut each external edge into two half-edges, thereby creating two leaves: the one at the end of the outgoing (resp. ingoing) half-edge is considered as an opening (resp. closing) leaf. The opening leaf at the end of the root half-edge is taken as the root leaf, see  Figure~\ref{fig:bij_eulerian}. The resulting tree $T'$ is easily checked to be an Eulerian tree, and it is endowed with a forward matching (as provided
 by the external edges). Conversely, for $T'$ an Eulerian tree endowed with a forward matching, we orient the edges of $T'$ toward the root, except for the edges leading to an opening leaf $o$, which we orient toward $o$. Then the matched leaves can be merged, each matched pair giving an external edge. The resulting map is clearly a rooted Eulerian map endowed with an Eulerian orientation $O$, and moreover if we let $T$ be the subtree of $T'$ induced by the nodes (i.e., excluding the leaves and their incident edges), then 
 $O=\Phi_M(T)$, hence $O$ is the minimal Eulerian orientation of $M$.    
 
\bigskip
To summarize, we obtain:
\begin{theo}
\label{thm:eul}
The following families are in bijection:
\begin{itemize}
\item{enriched Eulerian trees,}
\item{Eulerian trees endowed with a forward matching,}
\item{rooted Eulerian maps.}
\end{itemize}
The number of nodes of degree $2k$ in the first two families is preserved and corresponds to the number of vertices of degree $2k$ in the third family.
In particular the total half-degree $E$ of the nodes in the first two families, and of the vertices in the third family is preserved (this is also the number
of edges in the third family). The three families are
enumerated by $r_1(t)$ where $t$ is conjugate to $E$ and each node (resp. vertex) of degree $2k$ is weighted by $g_k$. 
\end{theo}
\vspace{.2cm}

\noindent{\bf Remark.} In the case where all matching-indices are $0$ (and only in that case) the associated rooted Eulerian map is planar, and our construction coincides with Schaeffer's bijection~\cite{Sc97} (in its reformulation relying on Eulerian orientations, as given in~\cite[Sect.3.1.2]{Fu07} and~\cite[Sect.3.1]{AlPo15}) between balanced Eulerian trees and rooted planar Eulerian maps. In the next section we will explain that, even in the unfixed genus case, one can see our construction as an application of Schaeffer's bijection. 

\subsection{A planar version of the bijection}\label{sec:planar}

\begin{figure}
\begin{center}
\includegraphics[width=14cm]{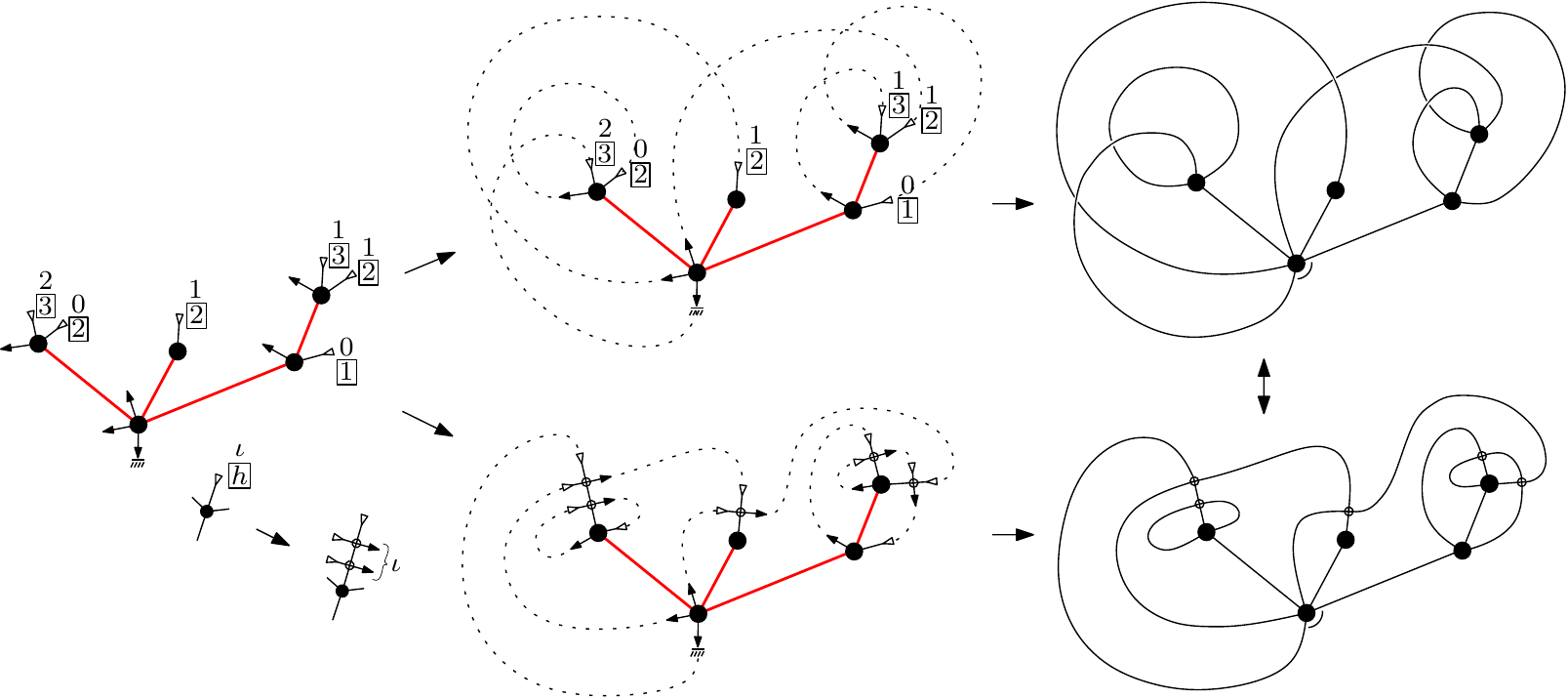}
\end{center}
\caption{{}From an enriched Eulerian tree to a rooted Eulerian map: the top-row shows the mapping as described in Section~\ref{sec:bij_eulerian}, the bottom-row shows the planar version.}
\label{fig:planar}
\end{figure}

As shown in Figure~\ref{fig:planar}, the bijection from enriched Eulerian trees to rooted Eulerian maps can alternatively be performed as follows. For $T$ an enriched Eulerian tree and $c$ a closing leaf of matching-index $\iota$, we call \emph{leaf-extension} the operation of replacing $c$ by a branch of length $\iota+1$ ending
 with a closing leaf, so that each of the $\iota$ internal vertices on the branch, which are called \emph{crossing-vertices},  carries a closing leaf on the left side and an opening leaf on the right side (seeing the branch as extended upward).
Let $T'$ be the balanced Eulerian tree obtained from $T$ after the leaf-extension of every closing leaf.  We may then perform Schaeffer's bijection to $T'$ (i.e., the bijection of Section~\ref{sec:bij_eulerian} where all closing leaves of $T'$ are considered to have matching-index $0$). What we obtain is a planar rooted Eulerian map $M'$ which exactly corresponds to $M$, upon seeing crossing-vertices as locations where two edges cross in the planar representation of $M$ (that $M'$ indeed yields $M$ in this reduction can be checked step-by-step when treating closing leaves in clockwise order around $T$, see Figure~\ref{fig:equiv}). 
This gives in particular a canonical planar representation of rooted Eulerian maps of any genus.

Note that if we let $[i]_q:=1+\cdots+q^{i-1}=\frac{1-q^i}{1-q}$ be the $q$-analog of $i$,  
 and let $r_i(t,q)$ be the series specified by the recursion relations\footnote{That is, $r_i(t,q)$ is the specialization of the series $\hr_i(t)$ given in 
\eqref{eq:recurhri} by setting $z_h=[h]_q=\sum\limits_{\iota=0}^{h-1}q^\iota$.}
\[
r_i(t,q)=[i]_q+\sum_{k\geq 1} t^k g_{k} \sum_{\wp \in \mathcal{P}_k^{(i)}} 
\prod_{\hbox{\tiny{descending steps}}\atop h\to h-1\ \hbox{\tiny{of}}\ \wp} r_h(t,q)\ , \quad i\geq 1\ ,
\] 
 then by the planar reformulation of the bijection, $r_1(q,t)$ is the counting series of rooted Eulerian maps with weight $q$ per crossing-vertex (i.e., the power of $q$ is the `crossing number' of the canonical planar representation of the map). This gives a unified formula covering both the planar case (by setting $q=0$)
and the unfixed genus case (by setting $q=1$). 

\vspace{.2cm}

\noindent{\bf Remark.} Our construction can thus be considered as an extension of Schaeffer's closure bijection~\cite{Sc97} to arbitrary rooted Eulerian maps, with control on a crossing-number parameter. This parameter does not seem to have a simple relation to the genus (except that it is zero iff the genus is zero). A different extension, with control on the genus, has been recently given in~\cite{Le19}: for any fixed genus $h$ it encodes a rooted Eulerian map of genus $h$ as a certain unicellular map of the same genus, endowed with a planar forward matching. On the other hand, the bijection for Eulerian planar maps based on labelled mobiles~\cite{BDG04} also extends to any fixed genus~\cite{CMS09,Ch09}, and we do not know if it could be given an alternative extension to unfixed genus explaining that rooted Eulerian maps are counted by $r_1(t)$.

\begin{figure}
\begin{center}
\includegraphics[width=12cm]{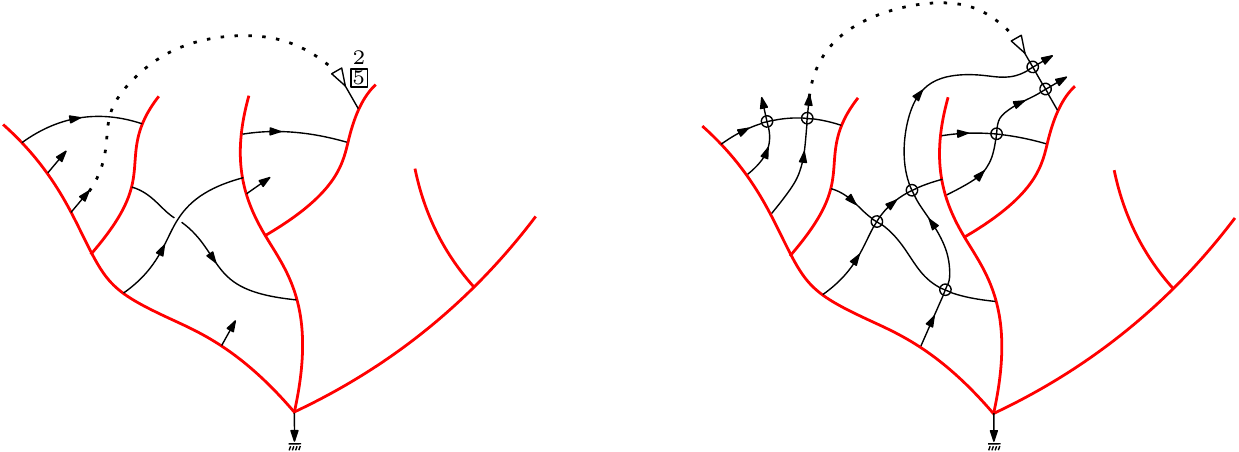}
\end{center}
\caption{The situation when matching a closing leaf (left-side in the first formulation, right-side in the planar reformulation). On the right, black edges
without arrows correspond to leaf-extensions.}
\label{fig:equiv}
\end{figure}

\section{More combinatorial results for Eulerian maps}
\label{sec:moreres}
We now extend the bijection of Theorem \ref{thm:eul} to Eulerian maps with marked edges, in correspondence with Eulerian trees endowed with a marked matching. This will allow us to give a combinatorial interpretation to the $r_i$'s for $i>1$ as counting series for particular ``admissible marked Eulerian maps with multiplicities". {}From 
\eqref{eq:Tvalbis}, we will then, after some manipulations, identify the (properly defined) counting series for marked maps without multiplicities with particular generating functions
for rooted face-colored Eulerian maps (Equation \ref{eq:bijtT}).

\subsection{Marked maps, marked matchings in blossoming trees}\label{sec:marked_maps}
We first extend to the so-called \emph{marked} setting the bijection between rooted Eulerian maps and Eulerian trees endowed with a forward matching. A \emph{marked map} $M$ is a rooted map where some edges are marked and oriented. An orientation of $M$ is called \emph{compatible} if it agrees with the fixed orientation of the marked edges. A compatible  orientation of $M$ is called \emph{root-accessible} if for each vertex $v$ there exists an oriented path from $v$ to the root vertex that avoids the marked edges. Let $V$ be the vertex-set of $M$, and $v_0$ the root vertex. A function $\alpha:V\to\mathbb{N}$ is called \emph{feasible} if there exists a compatible $\alpha$-orientation of $M$. It is called \emph{root-accessible} if there exists an $\alpha$-orientation that is compatible and root-accessible. In that case let $M'$ be the (unmarked) map obtained from $M$ by deleting the marked edges (the root corner of $M'$ is taken as the unique corner whose angular area contains the angular area of the root corner of $M$). Let $\alpha'$ be the function from $V$ to $\mathbb{N}$ such that for every $v\in V$, $\alpha'(v)$ is equal to $\alpha(v)$ minus the number of marked edges having $v$ as origin. Clearly the compatible $\alpha$-orientations of $M$ are in 1-to-1 correspondence with the $\alpha'$-orientations of $M'$. In addition the fact that $\alpha$ is feasible and root-accessible for $M$ ensures that $\alpha'$ is feasible and root-accessible for $M'$; this also ensures that every compatible $\alpha$-orientation of $M$ is root-accessible. The compatible $\alpha$-orientation $O$ associated to the minimal $\alpha'$-orientation $O'$ of $M'$ is called the \emph{canonical $\alpha$-orientation} of $M$. The spanning tree of $M'$ (and of $M$) such that $\Phi_{M'}(T)=O'$ is called the \emph{canonical spanning tree} of $M$ (for the function $\alpha$).

A marked Eulerian map is called \emph{admissible} if it admits an Eulerian orientation that is compatible, root-accessible, and where the root half-edge (which is possibly on a marked edge) is outgoing. 
For $a\geq 0$ we let $\cM_a$ be the set of admissible marked Eulerian maps having $a$ marked edges (note that $\cM_0$ is just the set of rooted Eulerian maps). 

\medskip
On the other hand, for $T$ a blossoming tree, a \emph{marked matching} of $T$ is a matching of the opening leaves with the closing leaves where some of the pairs are marked, and such that for each unmarked pair the opening leaf appears before the closing leaf in a clockwise walk around $T$ starting at the root.
A \emph{marked blossoming tree} is a blossoming tree endowed with a marked matching, For $a\geq 0$ we let $\cU_a$ be the set of marked Eulerian trees having $a$ marked pairs (note that $\cU_0$ is just the set of Eulerian trees endowed
with a forward matching). 

\begin{figure}
\begin{center}
\includegraphics[width=14cm]{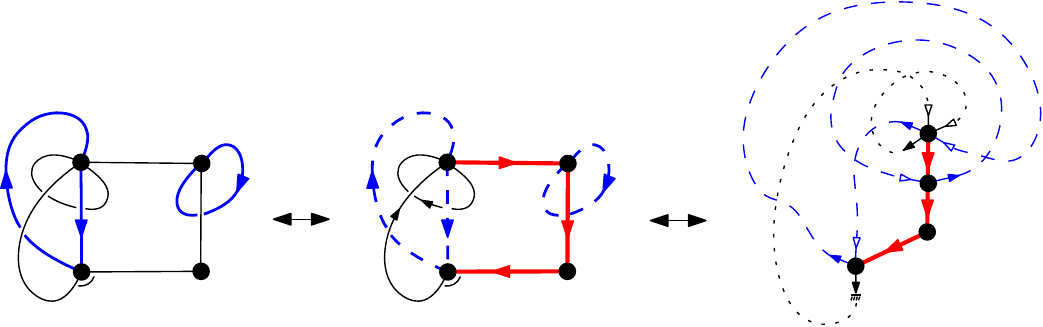}
\end{center}
\caption{Left: an admissible marked Eulerian map $M$ (with $3$ marked edges, shown blue). 
Middle: the canonical Eulerian orientation of $M$ (the edges of the canonical spanning tree are shown red).  Right: cutting each edge in its middle and taking as the root leaf the one at the end of the root half-edge, one obtains an Eulerian tree endowed with a marked matching (with $3$ marked matched pairs).}
\label{fig:bij_eulerian_marked}
\end{figure}

The bijection described in Section~\ref{sec:bij_eulerian} can be easily generalized as a bijection between $\cM_a$ and $\cU_a$, for any $a\geq 0$. Let $M\in\cM_a$, endowed with its canonical Eulerian orientation, and let $T$ be the canonical spanning  tree of $M$. 
Note that by Lemma~\ref{lem:outroot} the root half-edge of $M$ has to be outgoing (there is an easy case distinction whether the root edge of $M$ is marked or not). We can then cut all external edges (edges not in $T$, note that this includes all marked edges) at their middles, the end of the outgoing (resp. ingoing) half-edge being considered as an opening (resp. closing) leaf. 
The root leaf is taken as the opening leaf resulting from cutting the root edge of $M$ (since it goes out of $v_0$, the root leaf is adjacent to $v_0$).  What we obtain is clearly a marked Eulerian tree $T'$  having $a$ marked pairs (corresponding to the $a$ marked edges of $M$), see Figure~\ref{fig:bij_eulerian_marked}. 

Conversely, for $T'\in\cU_a$, very similarly as for $a=0$, we orient all edges of $T'$ toward the root, except for the edges incident to an opening leaf, which we orient toward the leaf. We then merge each matched pair of leaves into an (oriented) edge, which we consider as a marked edge if the pair is marked. We obtain an admissible 
 marked Eulerian map $M\in\cM_a$ endowed with its canonical Eulerian orientation, and such that the subtree $T$ of $T'$ induced by the nodes is the canonical spanning tree of $M$.   
       
\medskip
To summarize, we obtain:
\begin{prop}
The following families are in bijection for any $a\geq 0$:
\begin{itemize}
\item{marked Eulerian trees whose number of marked pairs is {\large$a$},}
\item{admissible marked Eulerian maps whose number of marked edges is {\large$a$}.}
\end{itemize}
The number of nodes of degree $2k$ in the first family corresponds to the number of vertices of degree $2k$ in the second family. 
\end{prop}

\subsection{Interpretation of $r_i(t)$ for general $i$}\label{sec:interpr_ri}
A \emph{marked map with multiplicities} 
is a marked map $M$ where each marked edge $e$ carries a multiplicity $\mu(e)\in\mathbb{N}^*$. 
For $i\geq 1$ we let $\cR_i$ be the family of admissible marked Eulerian maps with multiplicities, such that the total multiplicity, i.e. the sum of the 
multiplicities of all the marked edges, is (strictly) less than~$i$. We show here that $r_i(t)-i$ is the counting series of $\cR_i$. 

An \emph{$i$-enriched} blossoming tree is an $i$-balanced blossoming tree where each closing leaf $c$ of $i$-height $h$ 
carries an index $\iota(c)\in[0..h-1]$. 
As we have seen in Section~\ref{sec:euleriantree}, $r_i(t)-i$ is the counting series of $i$-enriched Eulerian trees with at least one node.   
 For such a tree $T$, using an operation quite similar to that in Section~\ref{sec:planar}, we may extend the root-leaf of $T$ into a branch $B$ of length $i$, such that $B$ ends with an opening leaf (the new root-leaf), and each of the $i-1$ internal vertices on the branch carries an opening leaf on the left side and a closing leaf on the right side (seeing the branch as extended downward, see Figure~\ref{fig:ri_map}). We let $T'$ be the Eulerian tree thus obtained. The $i-1$ internal vertices of $B$, and opening and closing leaves on each side, are called \emph{artificial}. Note that for each closing leaf of $T$, its $i$-height in $T$ becomes its height in $T'$, so that $T'$ is balanced. 

\begin{figure}
\begin{center}
\includegraphics[width=14cm]{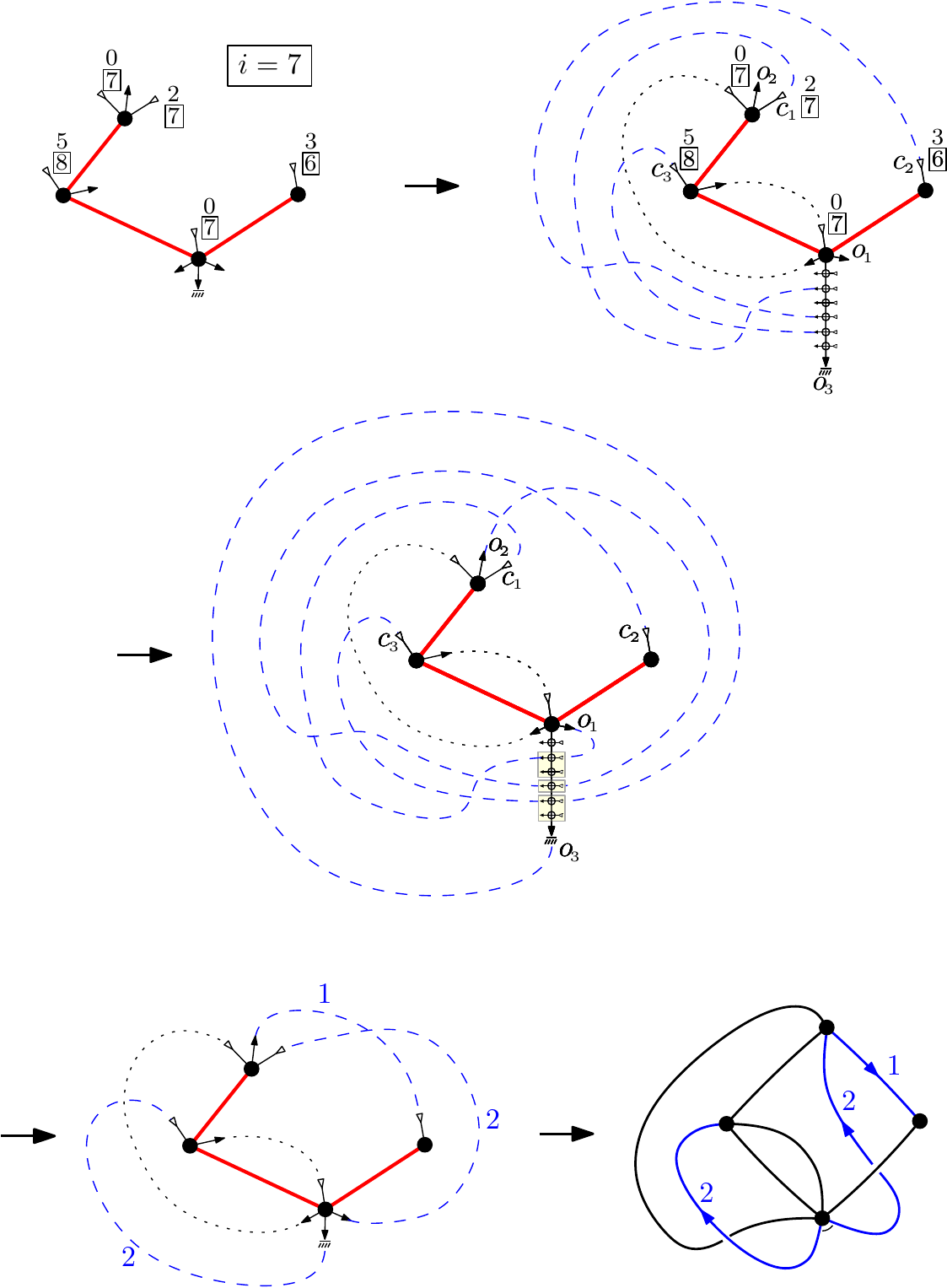}
\end{center}
\caption{{}From an $i$-enriched Eulerian tree ($i=7$ here) to a map in $\cR_i$.}
\label{fig:ri_map}
\end{figure}

Similarly as in Section~\ref{sec:bij_eulerian} we may then use the indices $\iota$ (one by each closing leaf) to construct a partial matching of the opening leaves with the closing leaves of $T'$ such that in each matched pair the opening leaf appears before the closing leaf in a clockwise walk around $T'$, and the closing leaves that are matched are only (and exactly) the non-artificial ones, see 2nd drawing in Figure~\ref{fig:ri_map}. Let $r\in[0..i-1]$ be the number of artificial opening leaves that are matched, and let $0<i_1<\cdots<i_r<i$ be their positions along $B$ (from top to bottom). For $j\in[1..r]$ let $c_j$ be the closing leaf matched with the  artificial opening leaf in position $i_j$.  
Note that $r$ is also the number of non-artificial opening leaves that are unmatched (indeed the total number of opening leaves that are unmatched is $i-1$). Let  $o_r,\ldots,o_1$ be these opening leaves, in their order of appearance around $T'$ (in particular $o_r$ is the root leaf iff it is unmatched). For $j\in[1..r]$ we can then match $o_j$ with 
the artificial closing leaf at position $i_j$ (see 3rd drawing). We may then erase all the artificial vertices, creating de facto a direct matching between $o_j$ and $c_j$,
which we mark and to which we give multiplicity $i_{j+1}-i_{j}$ (with the convention $i_{r+1}=i$, note that the sum of the $r$ multiplicities is $i-i_1\leq i-1$).  

We thus obtain a marked Eulerian tree in $\cU_r$ with multiplicities on the marked matched pairs that add up to less than $i$ (see 4th drawing). The corresponding admissible marked Eulerian map (having $r$ marked edges) with multiplicities is thus in $\cR_i$ (see 5th drawing). All steps of the construction can be inverted, so that we obtain:

\begin{prop}
\label{prop:ri}
The following families are in bijection for any $i\geq 1$:
\begin{itemize}
\item{$i$-enriched Eulerian trees,}
\item{admissible marked Eulerian maps with multiplicities adding up to less than $i$ ($\cR_i$).}
\end{itemize}
The number of nodes of degree $2k$ in the first family corresponds to the number of vertices of degree $2k$ in the second family.
In particular the total half-degree $E$ of the nodes in the first family and of the vertices in the second family is preserved (this is also the number
of edges in the second family). The two families are
enumerated by $r_i(t)$ where $t$ is conjugate to $E$ and each node (resp. vertex) of degree $2k$ is weighted by $g_k$. 
\end{prop}

The interpretation of $r_i(t)$ in terms of marked maps makes it possible to give a combinatorial proof of the identity \eqref{drivalue}, as detailed in 
Appendix \ref{sec:driproof}.

\begin{figure}
\begin{center}
\includegraphics[width=14cm]{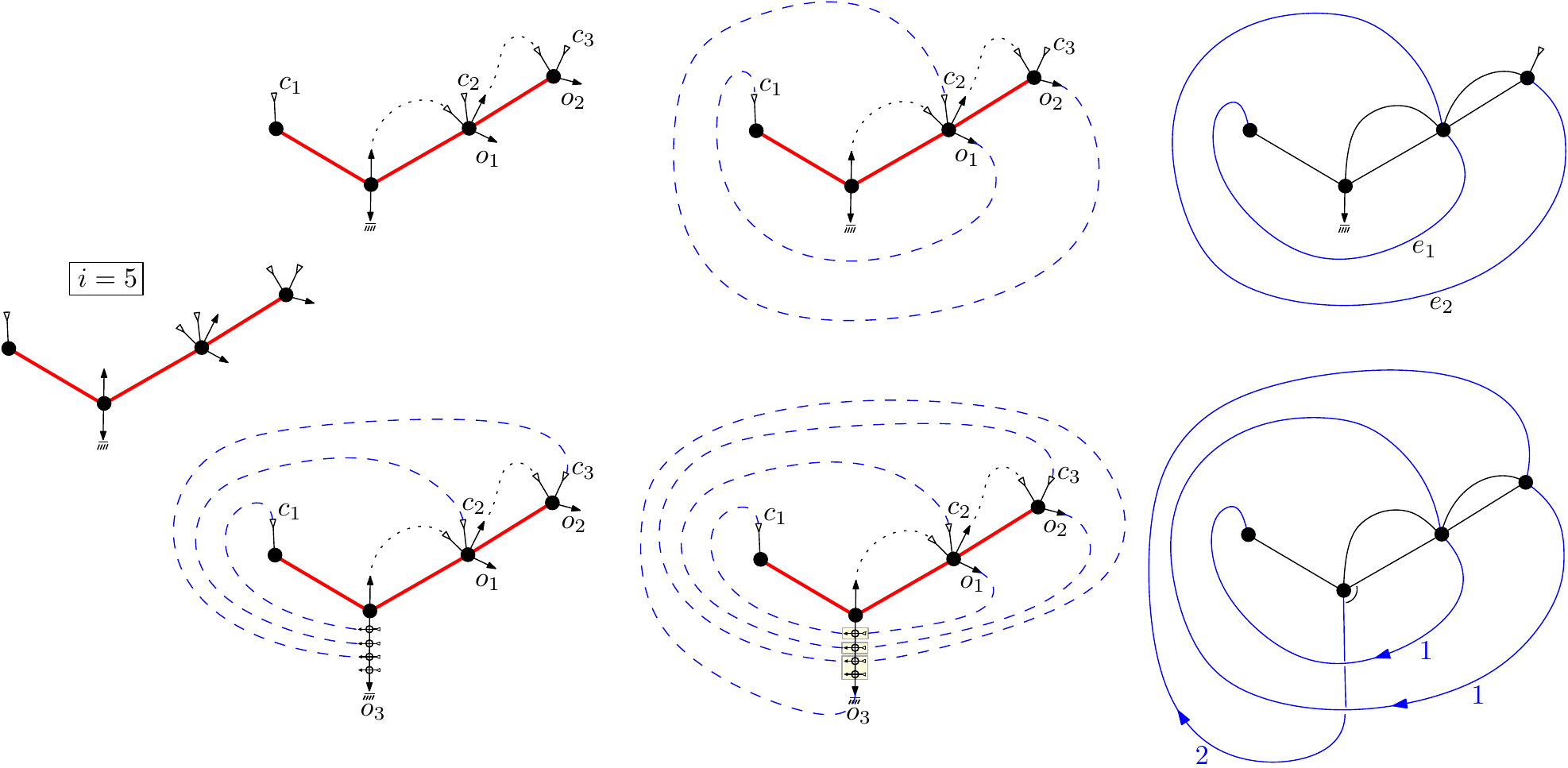}
\end{center}
\caption{Left: a $5$-balanced Eulerian tree (it is $i$-balanced for any $i\geq 3$). The top-row shows the closure bijection of~\cite{BDG03} to obtain a planar Eulerian map with two legs at distance $d\leq 4$ ($d=2$ here). The bottom-row shows that the same map (upon connecting the two legs to form the root edge) is obtained from our construction where all matching-indices are set to $0$.}
\label{fig:ri_map_planar}
\end{figure}

\medskip
\noindent{\bf Remark.} The above bijection extends, for any given $i\geq 1$, that of \cite{BDG03} between $i$-balanced Eulerian trees and two-leg planar Eulerian maps whose
two legs are at distance $d\leq i-1$ from each other (itself an extension of Schaeffer's bijection~\cite{Sc97}, which  corresponds to $i=1$). Here the distance $d$ is the minimal number of edges which need to be crossed to connect the two legs.
The two-leg map is obtained from the $i$-balanced tree $T$ as follows. We perform a clockwise walk around $T$ starting at the root. For each encountered closing leaf, we match
it to the first available opening leaf before it, if any. We obtain a partial forward matching leaving a number $2d+1$ of unmatched non-root leaves,
with $d\leq i-1$. More precisely, 
in clockwise order around $T$, the first unmatched leaves are $d+1$ closing leaves $c_1,\dots,c_{d+1}$, followed by $d$ unmatched non-root opening leaves
$o_d,\dots,o_1$. The construction is completed by matching $o_j$ to $c_j$ for $j=1,\dots,d$, leading to a planar map with two legs (one leading to $c_{d+1}$, 
the other to the root leaf) at distance $d$ from each other (see the top-row in Figure~\ref{fig:ri_map_planar}). Moreover via the bijection the map is naturally endowed with a canonical set of $d$ edges $e_1,\ldots,e_d$ separating the two legs, where each edge $e_j$ results from matching $o_j$ to $c_j$.  
Noting that $i$-balanced Eulerian trees are clearly identified with $i$-enriched Eulerian trees
where all the matching indices are $0$, the construction of \cite{BDG03} may be viewed, upon connecting the two legs to form the root edge of the map, 
as a specialization (see the bottom-row in Figure~\ref{fig:ri_map_planar}) of that of Proposition~\ref{prop:ri}. In our construction the non-root marked
edges are all of multiplicity~$1$ and they precisely correspond to the above mentioned edges $e_1,\ldots,e_d$ of the two-leg map. In addition, the root edge is marked with multiplicity $i-1-d$ if $d<i-1$ and unmarked for $d=i-1$ (in particular, the total multiplicity takes the maximal allowed value $i-1$).

\subsection{Face-colored maps}\label{sec:face-colored}
{}From \eqref{eq:Tvalbis} and Proposition~\ref{prop:ri}, we may interpret $T(t,N)$ as a particular counting series for marked Eulerian maps with multiplicities. 
As we will now show, this interpretation is more enlightening if we consider \emph{$N$-fully-colored} maps, i.e. face-colored maps with color set $\{1,2,\dots,N\}$ such that
for every $j\in[1..N]$ there is at least one face of color $j$. Let $\tT(t,N)$ be the counting series of $N$-fully-colored rooted Eulerian maps. Obviously we have
\begin{equation}
T(t,N)=\sum_{a=1}^N {N\choose a} \tT(t,a).
\label{eq:TTtilde}
\end{equation}
On the other hand, for $a\geq 1$ we let $u_a(t)$ be the counting series of $\cM_{a-1}$ (which is also the counting series of $\cU_{a-1}$ by the bijection of the previous section), and let $v_a(t)$ (resp. $w_a(t)$) 
be the series gathering the maps in $\cM_{a-1}$ where the root edge is unmarked (resp. marked).  Note that $v_{a}(t)=w_{a+1}(t)$ (by switching the status marked/unmarked of the root edge). In the previous section we have seen that $r_i(t)-i$ is the counting series of marked Eulerian trees with an unfixed number $a-1$ of marked
matched pairs carrying multiplicities adding up to less than~$i$. 
As we have seen, these multiplicities can be encoded by integers $0<i_1<\cdots<i_{a-1}<i$, so that 
\[
r_i(t)-i=\sum_{a=1}^{i}\binom{i-1}{a-1}u_a(t).
\] 
Hence if we let $U(x):=\sum_{a\geq 1} u_a(t)x^a$ then we have (using $\sum\limits_{i=a}^N\binom{i-1}{a-1}=\binom{N}{a}$)
\[
S(t,N):=\sum_{i=1}^N(r_i(t)-i)=\sum_{a=1}^N\binom{N}{a}u_a(t)=[x^N](1+x)^NU(x).
\]
We let $V(x):=\sum_{a\geq 1} v_a(t)x^a$, and note that $U(x)=(1+x)V(x)$ (since $v_a(t)=w_{a+1}(t)$ implies $U(x)-V(x)=xV(x)$).
We have from \eqref{eq:Tvalbis}
\begin{align*}
T(t,N)&=S(t,N)+S(t,N-1)\\
&=[x^N](1+x)^NU(x)+[x^{N-1}](1+x)^{N-1}U(x)\\
&=[x^N](1+x)^NU(x)+[x^{N}]x(1+x)^{N}V(x)\\
&=[x^N](1+x)^N(U(x)+xV(x))\\
&=\sum_{a=1}^N\binom{N}{a}\tilde{u}_a(t), \ \ \mathrm{with}\ \tilde{u}_a(t):=[x^a](U(x)+xV(x)).
\end{align*}
Note that $\tilde{u}_a(t)=u_a(t)+v_{a-1}(t)=u_a(t)+w_a(t)=v_a(t)+2w_a(t)$, i.e., $\tilde{u}_a(t)$ is the counting series of $\cM_{a-1}$ where a map is counted once if its root edge is unmarked and twice if it is marked (it is also the counting series of $\cU_{a-1}$ where a marked tree is counted once if the root leaf is not in a marked pair and twice otherwise). By comparing with \eqref{eq:TTtilde}, we obtain the remarkable identity
\begin{equation}\label{eq:bijtT}
\tT(t,N)=\tilde{u}_N(t),\ \ N\geq 1.
\end{equation}
In the case where there is a single vertex of degree $2n$, the constraint of being admissible is easily dealt with: a map in $\cM_{a-1}$ is completely encoded by the underlying one-vertex rooted map, by the choice of $a-1$ marked edges among the $n$ edges, and by a binary choice for each marked edge $e$ (if $e$ is not the root edge it gives the direction of $e$, if $e$ is the root edge its direction is fixed but we have to count the object twice, as mentioned above). Hence we have $[g_{n}t^{n}]\tilde{u}_a(t)=(2n-1)!!\binom{n}{a-1}2^{a-1}$, where the factor $(2n-1)!!$ gives the number of one-vertex rooted maps with $n$ edges. We thus recover the Harer-Zagier summation formula~\cite{HaZa86}
\begin{equation*}\label{eq:harer_zagier}
[g_{n}t^{n}]T(t,N)=(2n-1)!!\sum_{a=1}^N\binom{N}{a}\binom{n}{a-1}2^{a-1}.
\end{equation*}
It would be interesting to find a bijective proof of~\eqref{eq:bijtT} for $N\geq 2$. In the one-vertex case, bijective proofs of the Harer-Zagier summation formula (relying on the encoding of $N$-fully-colored one-vertex rooted maps) have been given in~\cite{GoNi05,Be12,ch13}. For more than one vertex, we note that it is not possible to find a  bijection for~\eqref{eq:bijtT} where the underlying graph is always preserved. Indeed, already for $N=2$ and two vertices of degree $4$, letting $\nu\in\{2,4\}$ be the number of edges connecting the two vertices, we find $[g_2^2t^4]\tT(t,2)=156$ with contribution $24$ when $\nu=4$ and contribution $132$ when $\nu=2$, whereas $[g_2^2t^4]\tilde{u}_2(t)=156$ with contribution $48$ when $\nu=4$ and contribution $108$ when $\nu=2$.   

\section{An analogous blossoming tree approach for $m$-regular bipartite maps with unfixed genus}
\label{sec:mregular}
The aim of this section is to transpose the above combinatorial correspondences between Eulerian maps and Eulerian trees to the
family of $m$-regular bipartite maps.
Recall that $m$-regular bipartite maps are maps where all vertices have degree $m$ and are colored in black or white so that no
two adjacent vertices have the same color. Such a map is called \emph{rooted} if it has a marked corner, called the \emph{root corner}, at a white vertex.
As before, the \emph{root vertex} is the one incident to the root corner, the \emph{root half-edge} is the half-edge following the root corner
in clockwise order around the root vertex, and the \emph{root edge} is the edge containing the root half-edge.
{}From now on, unless otherwise stated, we will consider $m\geq 3$. 
 
\subsection{Counting formulas from matrix integrals}
\label{sec:intmregular}
In this section, we are interested in the generating function $r_1(g)$ for rooted $m$-regular bipartite maps enumerated with a weight $g$ 
per black vertex (here again, all the maps that we consider are connected). As in Section~\ref{sec:recursion}, we may recourse to the 
appropriate integral representation of $m$-regular bipartite map generating 
functions to show that  $r_1(g)$ may be obtained as the first term of a family of functions $(r_i(g))_{i\geq 1}$ which are determined order by order in $g$
via a recursive system. The precise derivation of this statement is presented in Appendix \ref{sec:mregularint}, in the spirit of the analysis of
Section~\ref{sec:orthpol}, by use of bi-orthogonal polynomials (in the bipartite setting, the integrals as such are divergent but we can mimic them by formal 
operators acting on power series).
Let us summarize here the outcome of this derivation:
we define
\begin{equation*}
r_1(g):=1+\sum_{V_\bullet\geq 1}g^{V_\bullet} W_{V_\bullet}
\end{equation*} 
where  $W_{V_\bullet}$ is the generating function for rooted $m$-regular bipartite maps with a total of $V_\bullet$ black vertices. 
The function $r_1(g)$ is the first term of the family $(r_i(g))_{i\geq 1}$ determined by the recursive system
\begin{equation}
r_i(g)=i+ \sum_{a=0}^{m-2} q_{i-a}(g)\ , \qquad 
q_i(g)=g\, \prod_{a=0}^{m-2} r_{i+a}(g)\ , \quad i\geq 1
\label{eq:sysrq}
\end{equation}
with the convention that $q_{i}(g)=0$ for $i\leq 0$. In particular, for $i=1$, we have the simple relation
\begin{equation*}
r_1(g)=1+q_{1}(g)\ .
\end{equation*}

For $m=3$ the above recursive system reduces to 
 \begin{equation*}
r_i(g)=i+g\, r_i(g)\big(r_{i+1}(g)+r_{i-1}(g)\big)\ , \qquad i\geq 1
\end{equation*}
with the convention $r_0(g)=0$. At first orders in $g$, this yields
\begin{equation*}
r_i(g)=i+2i^2\, g+4i(2i^2+1)g^2+8i^2(5i^2+9)g^3+16i(14i^4+58i^2+15)g^4+\ldots
\end{equation*}
and in particular, from the $i=1$ series, 
\begin{equation*}
W_1= 2\ , \quad 
W_2= 12 \ , \quad 
W_3= 112 \ , \quad 
W_4= 1392\ . \quad 
\end{equation*}

Similarly to Section~\ref{sec:facecol}, for $m\geq 3$ we have the remarkable identity
\begin{equation}
m\, g \frac{d}{dg}{\rm Log}\, r_i(g)= q_{i+1}(g)-q_{i-m+1}(g)\ .
\label{eq:drimbis}
\end{equation}
This relation is proved by verification from the recursive system \eqref{eq:sysrq} in Appendix~\ref{sec:drim}, and bijectively in Appendix \ref{sec:proofdqim}. 

We may again extend $r_1(g)$ by considering the more general generating function $T(g,N)$ for faced-colored rooted $m$-regular bipartite maps, where the faces are colored 
with color set $\{1,2,\ldots,N\}$. {}From a matrix integral analysis analogous to that of Section~\ref{sec:facecol} (see~\cite[Sect.4.1]{DGZ95}), it is given by\footnote{As opposed to the $N=1$ case, some steps of the proof require genuine converging integrals rather than formal operators. This can be obtained by first allowing monochromatic edges,
with weights $c>1$, and then performing the limit $c\to 0$ via some analytic continuation.} 
\begin{equation*}
T(g,N)=N \big(r_1(g)-1\big)+\sum_{i=1}^N(N-i)\, m\, g\frac{d}{dg}{\rm Log}\,  r_i(g)
\end{equation*}
with the above counting series $r_i(g)$. {}From \eqref{eq:drimbis}, this simplifies into\footnote{This alternative expression for $T(g,N)$ may also be obtained directly
by inserting in the integrand of the matrix integral a factor that accounts for the root vertex, viewed as a marked white vertex with a natural ordering of its $m$ incident 
half-edges endowed by the root corner.} 
\begin{equation}
T(g,N)=\sum_{j=0}^{m-1} S(g,N-j)\ , \quad S(g,N):=\sum_{i=1}^N q_{i}(g)
\label{eq:TgN}
\end{equation}
with $q_i(g)$ as in \eqref{eq:sysrq}. 

\subsection{Bijection with blossoming trees}
\label{sec:blossommregular}
A tree is called \emph{bipartite} if its nodes are partitioned into white nodes and black nodes so that there is no edge connecting two white nodes or two black nodes. 
We define an \emph{$m$-bipartite tree} as a bipartite blossoming tree where all nodes have degree $m$, the root node (if the tree is not the nodeless one) is white, 
 all opening (resp. closing) leaves are adjacent to white (resp. black) nodes, and every white node has exactly one child that is a black node, the other $m-2$ children being opening leaves. It is easy to check that such a tree satisfies the blossoming tree property that there are as many opening leaves as closing leaves (indeed, if $n$ denotes the number of black nodes, then $n$ is also the number of white nodes since the node-to-parent mapping is a 1-to-1 correspondence from black nodes to white nodes; then the number of opening leaves is clearly $n(m-2)+1$ and the node-to-parent mapping applied this time to white nodes ensures that the number of closing leaves is $n(m-1)-(n-1)=n(m-2)+1$). 

For $i\geq 1$, we let $\hr_i(g)$ be the counting series of $i$-balanced $m$-bipartite trees with weight $g$ per black node and weight $z_h$ per closing leaf whose $i$-height is $h$. We let $\hq_i(g)$ be the series gathering the terms in $\hr_i(g)$ corresponding to trees that are not nodeless and such that the black-node child of the root node is its rightmost child. 
Note that for $a\in[0..m-2]$ the series gathering the terms of $\hr_i(g)$ where the black-node child of the root node
is the $(m-1-a)$th child (ordering children from left to right) is equal to $\hq_{i-a}(g)$, hence 
\[
\hr_i(g)=z_i+\sum_{a=0}^{m-2}\hq_{i-a}(g),\ \ \mathrm{with}\ \hq_j=0\ \mathrm{for}\ j<0.
\]
Next, a tree counted by $\hq_i(g)$ is decomposed at the black-node child of the root node into $m-1$ subtrees counted respectively, from left to right, by $\hr_{i+m-2}(g),\ldots,\hr_i(g)$. Hence 
\[
\hq_i(g)=g\, \prod_{a=0}^{m-2}\hr_{i+a}(g). 
\]

Similarly as for Eulerian trees, if we let $r_i(g)$ and $q_i(g)$ be the specializations of $\hr_i(g)$ and $\hq_i(g)$
 where $z_h$ has been set to $h$, then 
$r_1(g)$ (which is also $1+q_1(g)$) is the counting series of enriched $m$-bipartite trees, hence is also the counting series of $m$-bipartite trees endowed with a forward matching. 

\begin{lem}\label{lem:ori_mbip}
Let $m\geq 2$ and let $M$ be a bipartite map such that every vertex-degree is a multiple of $m$. 
Let $\alpha$ be the function from the vertex-set $V$ of $M$ to $\mathbb{N}$ such that for every black (resp. white) vertex $v$ of degree $md$ we have $\alpha(v)=d$ (resp. $\alpha(v)=(m-1)d$). Then $\alpha$ is feasible and for every vertex $v_0$ of $M$ it is $v_0$-accessible (i.e., every $\alpha$-orientation of $M$ is strongly connected).   
\end{lem}
\begin{proof}
Let $S\subseteq V$, and let $\Sw$ (resp. $\Sb$) be the set of white (resp. black) vertices in $S$. Let $\Ew$ (resp. $\Eb$)
be the set of edges of $M$ whose white extremity is in $\Sw$ (resp. whose black extremity is in $\Sb$). We have
\[
\alpha(S)=\alpha(\Sw)+\alpha(\Sb)=\big(1-\tfrac1{m}\big)|\Ew|+\tfrac1{m}|\Eb|.
\]
We clearly have $|\Ew|\geq|E_S|$ and $|\Eb|\geq|E_S|$, hence $\alpha(S)\geq |E_S|$, and moreover the inequality is tight iff $\Ew=E_S$ and $\Eb=E_S$, which happens iff $S=V$. Hence, if $S$ does not contain $v_0$ the inequality is strict, so that $\alpha$ satisfies the general criteria (as stated in Section~\ref{section:alpha}) ensuring that $\alpha$ is feasible and $v_0$-accessible.  
\end{proof}
In the specific case of $m$-regular bipartite maps, Lemma~\ref{lem:ori_mbip} ensures that these maps admit an orientation where white vertices have outdegree $m-1$ and black vertices have outdegree~$1$. 
Such orientations are called \emph{1-orientations}.

\begin{lem}\label{lem:mbiptree}
Let $M$ be a rooted $m$-regular bipartite map, let $O$ be its minimal 1-orientation, and let $T$ be the spanning tree of $M$ such that $\Phi_M(T)=O$. Then every white vertex has exactly one black child in $T$, and every external edge (edge of $M\backslash T$) is oriented from its white to its black extremity. In addition the unique ingoing edge at the root vertex $v_0$ is the one that precedes the root corner in clockwise order around $v_0$.  
\end{lem} 
\begin{proof}
White vertices have outdgree $m-1$, hence indegree $1$ and therefore have at most one child in $T$. Hence the mapping that sends a black vertex to its parent is injective from black vertices to white vertices. Since there is the same number of white vertices and black  vertices ($M$ being bipartite $m$-regular), the mapping is actually one-to-one, hence every white vertex has one child in $T$. This also ensures that all edges ingoing at a white vertex are in $T$, so that all external edges are oriented from their white to their black extremity. Let $e$ be the edge between the root vertex $v_0$ and its unique black child in $T$.  Let $h$ be the half-edge preceding the root corner in clockwise order around the root vertex. If $h$ was not on $e$ it would be outgoing and part of an external edge. But the opposite (ingoing) half-edge of $h$ would come before $h$ in a clockwise walk around $T$ starting at the root corner, giving a contradiction.  
\end{proof}

\begin{figure}
\begin{center}
\includegraphics[width=14cm]{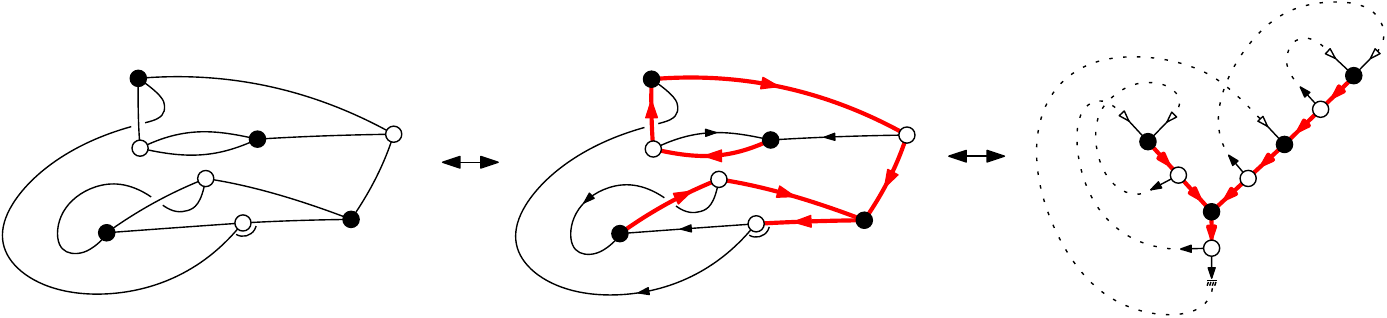}
\end{center}
\caption{Left: a rooted $m$-regular bipartite map $M$ (with $m=3$). Middle: the minimal 1-orientation of $M$ (with the associated spanning tree in red). Right: cutting each external edge in its middle and taking as the root leaf the extremity of the root half-edge, one obtains an $m$-bipartite tree endowed with a forward matching.}
\label{fig:bij_mbip}
\end{figure}

We can now describe a bijection (for any $m\geq 3$) between rooted $m$-regular bipartite maps and $m$-bipartite  trees endowed with a forward matching. For $M$ a rooted $m$-regular bipartite map, with $O$ its minimal 1-orientation and $T$ the spanning tree such that $\Phi_M(T)=O$, we cut each external edge (edge of $M\backslash T$) at its middle, thereby creating two edges, the end of the outgoing (resp. ingoing) half-edge being considered as an opening (resp. closing) leaf. 
The root leaf is taken as the opening leaf resulting from cutting the root edge of $M$ (which has to be external according to the last point in Lemma~\ref{lem:mbiptree}).   
 We clearly obtain an $m$-bipartite tree $T'$ endowed with a forward matching (a matched pair for each cut edge), see Figure~\ref{fig:bij_mbip}. 

Conversely, for $T'$ an $m$-bipartite tree endowed with a forward matching, we orient 
all edges of $T'$ toward the root, except for the edges incident to an opening leaf, which we orient toward the leaf. 
We then merge each matched pair of leaves into an edge. We obtain a rooted $m$-regular bipartite map $M$ endowed with a 1-orientation $O$. In addition, if we let $T$ be the subtree of $T'$ induced by the nodes, then we have $\Phi_M(T)=O$, so that $O$ 
is the minimal 1-orientation of $M$. 

\medskip
To summarize, we obtain:
\begin{theo}
\label{thm:bip}
The following families are in bijection, for any $m\geq 3$:
\begin{itemize}
\item{enriched $m$-bipartite trees,}
\item{$m$-bipartite trees endowed with a forward matching,}
\item{rooted $m$-regular bipartite maps.}
\end{itemize}
The number of black nodes in the first two families is preserved and corresponds to the number of black vertices in the third family.
The three families are
enumerated by $r_1(g)$ where $g$ is conjugate to the number of black nodes (resp. vertices). 
\end{theo}
\vspace{.2cm}

\noindent{\bf Remark.} As in the Eulerian case, the $m$-regular bipartite map associated to an enriched $m$-bipartite  tree is planar iff all matching indices are $0$. In that case our construction coincides with the bijection  by Bousquet-M\'elou and Schaeffer~\cite{BoSc00}, in its reformulation relying on 1-orientations 
 as given in~\cite[Sect.3.2]{AlPo15} (these bijections hold more generally for bipartite maps where white vertices have degree $m$ and black vertices have degrees multiple of $m$).  As in Section~\ref{sec:planar}, our construction can be reformulated as applying the planar case bijection, upon performing leaf-extension operations at closing leaves.

\subsection{More results using blossoming trees}
\label{sec:morem}
Similarly as for Eulerian maps, we can obtain more combinatorial results using the setting of marked maps (for convenience we override the analogous notation used for Eulerian maps). 
A marked $m$-regular bipartite map is called \emph{admissible} if it admits a compatible root-accessible 1-orientation where the unique ingoing edge at the root vertex $v_0$ is the one preceding the root corner in clockwise order around $v_0$. Let $\cM_a$ be the family of admissible marked $m$-regular bipartite maps with $a$ marked edges. 
On the other hand, let $\cU_a$ be the family of marked $m$-bipartite trees with $a$ marked pairs, and such that the unique node-child of the root node is the rightmost child. Note that the bijection of the previous section is from $\cM_0$ to $\cU_0$. 
For $M\in\cM_a$, let $O$ be the canonical 1-orientation of $M$, and let $T$ be its canonical spanning tree. By Lemma~\ref{lem:outroot} the unique ingoing edge at the root vertex $v_0$ is the one preceding the root corner in clockwise order around $v_0$. Let $T'\in\cU_a$ be the $m$-bipartite tree obtained from $M$ by cutting each external edge (edge not in $T$) at its middle, the end of the outgoing (resp. ingoing) half-edge being considered as an opening (resp. closing) leaf. The pairs resulting from marked edges (which have to be external) are declared as marked pairs, 
and the root leaf is taken as the opening leaf resulting from cutting the root edge of $M$ (since it goes out of $v_0$, the root leaf is adjacent to $v_0$, note that the root edge is possibly marked, in which case the pair involving the root leaf is marked). Conversely, for $T'\in\cU_a$, we orient all edges of $T'$ toward the root, except for the edges incident to the opening leaves, which we orient toward the leaf. We then merge each matched pair of leaves into an (oriented) edge, which we consider as a marked edge if the pair is marked. We obtain an admissible 
 marked $m$-regular bipartite map $M\in\cM_a$ endowed with its canonical 1-orientation, and such that the subtree $T$ of $T'$ induced by the nodes is the canonical spanning tree of $M$.

Thus, very similarly as in Section~\ref{sec:marked_maps} (Figure~\ref{fig:bij_eulerian_marked} for Eulerian maps), we obtain:

\begin{prop}
The following families are in bijection, for any $m\geq 3$ and $a\geq 0$:
\begin{itemize}
\item
marked $m$-bipartite trees whose number of marked pairs is {\large$a$}, such that the unique node-child of the root node is the rightmost child,
\item
admissible marked $m$-regular bipartite maps whose number of marked edges is {\large$a$}.
\end{itemize}
The number of black nodes in the first family corresponds to the number of black vertices in the second family.
\end{prop}

If we now consider marked maps with multiplicities, then a construction similar to that of Section~\ref{sec:interpr_ri} (Figure~\ref{fig:ri_map} for Eulerian maps) yields:

\begin{prop}
For $i\geq 1$ and $m\geq 3$, the following families are in bijection:
\begin{itemize}
\item
$i$-enriched $m$-bipartite trees such that the unique node-child of the root node is its rightmost child,
\item
the family, denoted by $\cQ_i$, of admissible marked $m$-regular bipartite maps with multiplicities  adding up to less than $i$. 
\end{itemize}
The number of black nodes in the first family corresponds to the number of black vertices in the second family. The counting series of both families is $q_i(g)$, with $g$ conjugate to the number of black nodes (resp. black vertices). 
\label{prop:bij_qi}
\end{prop}
The interpretation of $q_i(g)$ in terms of marked maps allows us to obtain a combinatorial proof of the identity \eqref{eq:drimbis}, as detailed in 
Appendix \ref{sec:proofdqim}.

\medskip

We now look at the analogue of the results of Section~\ref{sec:face-colored} regarding face-colored maps. Let $\tT(g,N)$ be the counting series of $N$-fully-colored rooted $m$-regular bipartite maps, related to $T(g,N)$ by the relation 
\begin{equation}
T(g,N)=\sum_{a=1}^N\binom{N}{a}\tT(g,a).
\label{TTtildem}
\end{equation}

For $a\geq 1$ we let $u_a(g)$ be the counting series of maps in $\cM_{a-1}$ and let $v_a(g)$ be the counting series for those with no marked edge incident to the root vertex. We also let $U(x):=\sum_{a\geq 1} u_a(g)x^a$ and $V(x):=\sum_{a\geq 1} v_a(g)x^a$.
\begin{lem}\label{lem:em}
Let $M$ be a marked admissible $m$-regular bipartite map, with $e_0,\ldots,e_{m-1}$ the edges incident to the root vertex $v_0$ in clockwise order,  starting from the root corner. Then $e_{m-1}$ can not be marked, and every marked edge of $M$ incident to $v_0$ is oriented out of $v_0$. 

Conversely, if $M$ has no marked edge, let $X$ be an arbitrary subset of $\{e_0,\ldots,e_{m-2}\}$, and let $M'$ be the map obtained from $M$ by additionally marking the edges in $X$ and orienting them out of $v_0$. Then $M'$ is admissible.  
\end{lem}
\begin{proof}
By definition $M$ admits a compatible root-accessible 1-orientation $O$ where the only ingoing edge at the root vertex $v_0$ is $e_{m-1}$. 
Since $O$ is root-accessible, the edge $e_{m-1}$ has to be unmarked. 
Moreover, for any $r\in[0..m-2]$ if $e_r$ is marked then its fixed orientation is out of $v_0$, and if $e_r$ is unmarked, then we can declare it as marked (and oriented out of $v_0$) and the orientation will clearly still be root-accessible.  
\end{proof}

 It follows from Lemma~\ref{lem:em} that $U(x)=(1+x)^{m-1}V(x)$. Moreover we have from Proposition~\ref{prop:bij_qi} that $q_i(g)=\sum_{a\geq 1}\binom{i-1}{a-1}u_a(g)$, hence
\[
S(g,N):=\sum_{i=1}^Nq_i(g)=\sum_{a=1}^N\binom{N}{a}u_a(g)=[x^N](1+x)^NU(x)=[x^N](1+x)^{N+m-1}V(x).
\]
Hence, from \eqref{eq:TgN}
\begin{align*}
T(g,N)&=\sum_{r=0}^{m-1}S(g,N-r)\\
&=\sum_{r=0}^{m-1}[x^{N-r}](1+x)^{N-r+m-1}V(x)\\
&=[x^N](1+x)^N\sum_{r=0}^{m-1}x^r(1+x)^{m-1-r}V(x)\\
&=\sum_{a=1}^N\binom{N}{a}\tilde{u}_a(g),\ \ \mathrm{with}\ \tilde{u}_a(g):=[x^a]\sum_{r=0}^{m-1}x^r(1+x)^{m-1-r}V(x).
\end{align*}
With the notation of Lemma~\ref{lem:em}, a marked admissible $m$-regular bipartite map is said to have \emph{root-index} $p\in[1..m]$ if $p$ is the smallest index such that $e_{p-1}$ is unmarked. By Lemma~\ref{lem:em}, for $r\in[0..m-1]$, $[x^a]x^r(1+x)^{m-1-r}V(x)$ is the counting series of maps in $\cM_{a-1}$ whose root-index is larger than $r$. Hence $\tilde{u}_a(g)$ is the counting series of $\cM_{a-1}$ where every map with root-index
$p$ is counted $p$ times. Comparing with \eqref{TTtildem}, we obtain the remarkable identity (analogue of~\eqref{eq:bijtT})
\begin{equation}\label{eq:bijtTm}
\tT(g,N)=\tilde{u}_N(g),\ \ N\geq 1,
\end{equation}
for which a bijection is to be found for $N\geq 2$. In the case where there are just two vertices each of degree $m$, for $a\in[1..m]$ and $r\in[0..m-1]$ the number of maps in $\cM_{a-1}$ with root-index larger than $r$ is clearly $m!\binom{m-1-r}{a-1-r}$ (there are $m!$ possibilities for the underlying rooted map, and with the notation of Lemma~\ref{lem:em}, the edges $e_0,\ldots,e_{r-1}$ are marked, and one has to choose $a-1-r$ marked edges among $e_r,\ldots,e_{m-2}$).  
Hence 
\[
[g]\tilde{u}_a(g)=m!\sum_{r=0}^{m-1}\binom{m-1-r}{a-1-r}=m!\binom{m}{a-1}. 
\]
We thus obtain
\begin{equation*}
[g]T(g,N)=m!\sum_{a=1}^{m}\binom{N}{a}\binom{m}{a-1},
\end{equation*}  
a special case of a counting formula of Goulden and Slofstra~\cite{GoSl10} for face-colored rooted maps with two vertices (not necessarily bipartite), which they prove bijectively. 
 
\section{Other results}
\label{sec:conclusion}
This section starts with a brief discussion on how our bijections between rooted Eulerian maps with unfixed genus and decorated Eulerian trees
can be extended to maps with arbitrary vertex degrees. We then explain how the recursive systems combined with differential identities make it possible to automatically obtain non-linear differential equations for counting series of rooted maps of unfixed genus and bounded vertex degrees. We finally give for small vertex-degrees a unified expression of the counting series as continued fractions.
\subsection{Maps with arbitrary vertex degrees}
\label{sec:3regular}
The case of maps with vertices of arbitrary degrees can also be described in terms of blossoming trees. Here to keep formulas simple,
we focus on the case of $3$-regular maps. Let us start by listing, without proofs, the 
results of the (matrix-)integral formulation of its generating function. Denoting by $M_3(g)$ the generating function for rooted $3$-regular maps 
with a weight $g$ per vertex, we have 
\begin{equation}
M_3(g)=r_1(g)+s_0^2(g)-1
\label{eq:M3formula}
\end{equation}
where $r_1(g)$ and $s_0(g)$ are the first terms of two families of counting series $r_i(g)$, $i\geq 1$ and $s_i(g)$, $i\geq 0$
determined order by order in $g$ by the infinite recursive system
\begin{equation*}
\begin{split}
r_i(g)&=i+g\, r_i(g)\big(s_i(g)+s_{i-1}(g)\big)\ ,\quad i\geq 1\\
s_i(g)&=g\, \big(r_{i+1}(g)+r_i(g)+s_i^2(g)\big)\ , \quad i\geq 0\\
\end{split}
\end{equation*}
with the convention $r_0(g)=0$.
As easily shown from these recursion relations, we have in particular
\begin{equation*}
\begin{split}
3g\frac{d}{dg}{\rm Log}\, r_i(g)&=\big(r_{i+1}(g)-r_{i-1}(g)-2\big)+\big(s_i^2(g)-s_{i-1}^2(g)\big)\ , \quad i\geq 1\\
3g^2\frac{d}{dg}s_i(g)&=\big(r_{i+1}(g)-r_{i}(g)-1\big)-g\, s_i(g)\ , \quad i\geq 0\ . \\
\end{split}
\end{equation*}
The generating function for face-colored rooted $3$-regular maps reads then
\begin{equation*}
\begin{split}
T(g,N)&=N\big(r_1(g)+s_0^2(g)-1\big)+\sum_{i=1}^{N}(N-i)\, 3g\frac{d}{dg}{\rm Log}\, r_i(g)\\
&=(r_N(g)-N)+2\sum_{i=1}^{N-1}(r_i(g)-i)+\sum_{i=0}^{N-1}s_i^2(g)\ .\\
\end{split}
\end{equation*} 

The functions $r_1(t)$ and $s_0(t)$ may be easily interpreted as counting series for appropriate enriched blossoming trees.
As before, we may design a bijection between rooted $3$-regular maps and the same blossoming trees endowed with a forward matching.
This yields a bijective interpretation of \eqref{eq:M3formula}. 
A key ingredient of the bijection from maps to trees consists in doubling the edges so that all vertices now have degree $6$ and we can use the minimal 
Eulerian orientation of the obtained map.
Such an approach was already used in \cite[Sect.3.1]{AlPo15} in the planar case and for arbitrary degrees.

\subsection{Non-linear differential equations and continued fractions}
\label{Contfrac}
In this section we first show that the differential identity~\eqref{drivalue} combined with the equation for $r_1(t)$ makes it possible to  automatically obtain non-linear differential equations for the counting series of rooted Eulerian maps of bounded vertex-degrees. We then show how the strategy can be adapted for $m$-regular bipartite maps, and discuss the occurence of simple continued fraction expansions for small vertex-degrees, where the differential equations are first order, of the Riccati type.
 
The identity~\eqref{drivalue} is equivalent to
\begin{equation*}
r_{i+1}(t)=\frac{2t}{r_i(t)}\frac{d}{dt}r_i(t)+r_{i-1}(t)+2,\ \ i\geq 1,
\end{equation*}
which holds for arbitrary weights $g_k$ per vertex of degree $2k$ (and does not depend on these weights). This identity ensures inductively that for $i\geq 2$, $r_i(t)$ admits a rational expression in terms of $t,r_1(t),\ldots,\frac{d^{i-1}}{dt^{i-1}}r_1(t)$, which we call the $r_1$-expression of $r_i(t)$.

If we consider now, for an arbitrary $b\geq 1$, Eulerian maps with a bound $2b$ on the vertex-degree, and weight $g_k$ per vertex of degree $2k$ for $k\in[1..b]$, then the equation for $r_1(t)$ in~\eqref{eq:recurri} is
\begin{equation}\label{eq:r1}
r_1(t)=1+\sum_{k= 1}^b t^k g_{k} \sum_{\wp \in \mathcal{P}_k^{(1)}} 
\prod_{\hbox{\tiny{descending steps}}\atop h\to h-1\ \hbox{\tiny{of}}\ \wp} r_h(t),
\end{equation}
which gives an algebraic equation relating $r_1(t),\ldots,r_b(t)$ (and also involving the parameters $t,g_1,\ldots,g_b$). If in~\eqref{eq:r1} we replace each of $r_2(t),\ldots,r_b(t)$ by its $r_1$-expression, we obtain a non-linear differential equation of order $b-1$ for $r_1(t)$, which is here the counting series of rooted Eulerian maps with a weight $t$ per edge, weight $g_k$ per vertex of degree $2k$ for $k\in[1..b]$, and no vertex of degree larger than $2b$.  

Let us derive this equation in the
simple cases of $4$-regular maps (i.e., $b=2$ and $g_k=\delta_{k,2}$) and $6$-regular maps (i.e., $b=3$ and $g_k=\delta_{k,3}$).

\medskip
For 4-regular maps the equation~\eqref{eq:r1} is 
\begin{equation*}
r_1(t)=1+t^2\, r_1(t)\big(r_{2}(t)+r_1(t)\big)\ .
\end{equation*}
Replacing $r_2(t)$ by its $r_1$-expression $\frac{2t}{r_1(t)}\frac{d}{dt}r_1(t)+2$
leads to the non-linear first order differential equation
\begin{equation*}
r_1(t)=1+2t^2\, r_1(t)+t^2\, \big(r_1(t)\big)^2 + 2t^3\frac{d}{dt} r_1(t)\ .
\end{equation*}
Introducing the generating function $M_4(g)$ for rooted $4$-regular maps with a weight $g$ per vertex, we have 
$M_4(g)=r_1(t)-1$ with $t=\sqrt{g}$ since $4$-regular maps with $E$ edges have $E/2$ vertices. {}From the above equation,
we deduce immediately that
\begin{equation}
M_4(g)=3g\, +4g\, \left(M_4(g)+g \frac{d}{dg}M_4(g)\right)+g\, \big(M_4(g)\big)^2
\label{eq:M4}
\end{equation}
which determines uniquely $M_4(g)$ as a power series in $g$.
 
\medskip
For $6$-regular maps,  the equation~\eqref{eq:r1} is 
\begin{equation*}
r_1(t)=1+t^3\, r_1(t)\Big(r_{2}(t)\big(r_{3}(t)+r_{2}(t)\big)+r_1(t)\big(2r_{2}(t)+r_1(t)\big)\Big)\ .
\end{equation*}
Replacing $r_2(t)$ and $r_3(t)$ by their $r_1$-expressions and rearranging leads now to the second order non-linear differential equation 
\begin{equation*}
\!\!\!\!\!\!\!\!\! r_1(t)=1+8t^3\, r_1(t)+6t^3\, \big(r_1(t)\big)^2 +t^3\, \big(r_1(t)\big)^3+ 16t^4\frac{d}{dt} r_1(t)
+ 6t^4 r_1(t)\frac{d}{dt} r_1(t)+4t^5\frac{d^2}{dt^2}r_1(t)\ .
\end{equation*}
In terms of the generating function $M_6(g)=r_1\left(g^{1/3}\right)-1$ for rooted $6$-regular maps with a weight $g$ per vertex,
this equation  may be rewritten as
\begin{equation*}
\begin{split}
M_6(g)&=15g\, +23g M_6(g)+9g\, \big(M_6(g)\big)^2+ g\, \big(M_6(g)\big)^3+90g^2 \frac{d}{dg}M_6(g)
\\ & \qquad \qquad \qquad +18g^2\, M_6(g) \frac{d}{dg}M_6(g) +36 g^3\frac{d^2}{dg^2}M_6(g)
\end{split}
\end{equation*}
which determines $M_6(g)$ as a power series in $g$.

We can follow a quite similar strategy for $m$-regular bipartite maps, for any $m\geq 3$. The recursive system~\eqref{eq:sysQ} gives
\[
r_i(g)=i+g\sum_{a=0}^{m-2}\prod_{j=0}^{m-2}r_{i-a+j}(g),\ \ i\geq 1,
\]
which leads to
\[
r_{i+m-2}(g)=\frac{1}{g\prod\limits_{j=0}^{m-3}r_{i+j}(g)}\big(r_i(g)-i-g\sum_{a=1}^{m-2}\prod_{j=0}^{m-2}r_{i-a+j}(g)\big),\ \ i\geq 1.
\]
By induction on $i\geq m-1$ it ensures that $r_i(g)$ admits a rational expression in terms of $g,r_1(g),\ldots,r_{m-2}(g)$,  which we call the \emph{$r$-expression} of $r_i(g)$. Now the differential identity~\eqref{eq:drimbis} gives the $m-2$ equations
\[
m\frac{d}{dg}r_a(g)=\prod_{j=0}^{m-1}r_{a+j}(g)\ \ \mathrm{for}\ a\in[1..m-2].
\]
If in these equations we replace each occurence of $r_j(g)$ (for $j\geq m-1$) by its $r$-expression, then we obtain a system of $m-2$ first order non-linear differential equations on the series $r_1(g),\ldots,r_{m-2}(g)$, of the form
\begin{equation*}\label{eq:dra}
m\frac{d}{dg}r_a(g)=F_a(g,r_1(g),\ldots,r_{m-2}(g))\ \ \mathrm{for}\ a\in[1..m-2],
\end{equation*}
where each $F_a$ is an explicit rational expression. 
For example, for $m=3$, the steps are the following. The $r$-expression of $r_2(g)$  
is extracted from the equation $r_1(g)=1+ g\, r_1(g)\, r_2(g)$, and then the $r$-expression of $r_3(g)$ is extracted from 
 the equation $r_2(g)=2+gr_2(g)\big(r_3(g)+r_1(g)\big)$; then these expressions are substituted in the differential equation $3\frac{d}{dg}r_1(g)=r_1(g)r_2(g)r_3(g)$, which after rearranging 
leads to the first order non-linear differential equation
\begin{equation*}
r_1(g)=1+g\, r_1(g)+g\, \big(r_1(g)\big)^2 + 3 g^2 \frac{d}{dg} r_1(g)\ ,
\end{equation*}
or equivalently, setting $M_{3{\rm b}}(g)=r_1(g)-1$, to 
\begin{equation}
M_{3{\rm b}}(g)=2g+3g\, \Big(M_{3{\rm b}}(g)+g \frac{d}{dg}M_{3{\rm b}}(g)\Big)+g\, \big(M_{3{\rm b}}(g))^2 \ .
\label{eq:M3b}
\end{equation}
hence an equation very similar to \eqref{eq:M4}.

Finally, for rooted $3$-regular maps, the relations of Section~\ref{sec:3regular} lead to
\begin{equation}
M_3(g)=5g^2+6g^2\, \Big(M_3(g)+g^2 \frac{d}{dg^2}M_3(g)\Big)+g^2\, \big(M_3(g))^2
\label{eq:M3}
\end{equation}
(note that $M_3(g)$ is actually a power series in $g^2$).

\medskip
Equations \eqref{eq:M4}, \eqref{eq:M3b} and \eqref{eq:M3} all take the form (of the Riccati type)
\begin{equation*}
M(x)=(p-1)\, x+p\, x\, \Big(M(x)+x\frac{d}{dx}M(x)\Big)+x\, \big(M(x))^2
\end{equation*}
with respectively $M=M_4$, $M_{3\rm{b}}$ and $M_3$, $x=g$, $g$ and $g^2$, and where\footnote{The value $p=2$ is also of interest and gives 
the number of rooted trivalent maps where the edges are colored blue, green red in clockwise order around each vertex
and the root edge is blue.}$p=4$, $3$ and $6$. Note that each of these equations has a unique
solution which is a formal power series in $x$.
Remarkably, the solution of this equation is a simple continued fraction
\begin{equation*}
\begin{split}
1+M(x)=\frac{1}{1-}\overbracket{\frac{(p-1)\, x}{1-}\frac{(p+1)x}{1-}}&\overbracket{\frac{(2p-1)x}{1-}\frac{(2p+1)x}{1-}}\overbracket{\frac{(3p-1)x}{1-}\frac{(3p+1)x}{1-}}\cdots\\
&
\overbracket{\frac{(k\,p-1)x}{1-}\frac{(k\, p+1)x}{1-}}\cdots\\
\end{split}
\end{equation*}
with a regular pattern of length $2$ indexed by the set of increasing integers $k$ as shown. To prove this statement, we follow the approach of \cite{AB00} 
where a similar differential equation was discussed in the context of arbitrary maps enumerated by their number of edges (it would also be possible to use~\cite[Theo.2.2]{me62} which provides a general statement to obtain the continued fraction expansion of a solution to a differential equation of the Riccati type, which the authors apply in~\cite[Theo.3.2]{me62} to an equation similar to ours).   
We start by introducing the (unique) power series $A_k(x)$ solution of the equation 
\begin{equation*}
A_k(x)=(k\, p-1)\, x+p\, x\, \Big((2k-1)A_k(x)+x\frac{d}{dx}A_k(x)\Big)+x\, \big((k-1)\, p+1\big)\big(A_k(x))^2\ .
\end{equation*} 
Then clearly $M(x)=A_1(x)$. Consider then the quantity $B_k(x)$ defined by 
\begin{equation*}
1+A_k(x)=\frac{1}{1-(k\, p-1)\, x\, \big(1+B_k(x)\big)}
\end{equation*} 
which is a power series in $x$. It is easily checked that the above differential equation for $A_k(x)$ implies that $B_k(x)$ is solution
of 
\begin{equation*}
B_k(x)=(k\, p+1)\, x+p\, x\, \Big(2k\, B_k(x)+x\frac{d}{dx}B_k(x)\Big)+x\, \big(k\, p-1\big)\big(B_k(x))^2\ .
\end{equation*} 
Introduce then the quantity $C_k(x)$ defined by 
\begin{equation*}
1+B_k(x)=\frac{1}{1-(k\, p+1)\, x\, \big(1+C_k(x)\big)}
\end{equation*} 
which is a power series in $x$.
Then the equation for $B_k(x)$ implies that $C_k(x)$ is solution of \emph{the same differential equation} as $A_k(x)$ up to a shift $k\to k+1$.
We immediately deduce that $C_k(x)=A_{k+1}(x)$ and therefore
\begin{equation*}
1+A_k(x)=\frac{1}{1-\frac{(k\, p-1)\, x\,}{1-(k\, p+1)\, x\, \big(1+A_{k+1}(x)\big)}}\ .
\end{equation*}
The continued fraction form above for $M(x)=A_1(x)$ follows immediately. It would be nice to have a simple combinatorial explanation for the resulting
simple expressions for $M_4(g)$, $M_{3\rm{b}}(g)$ and $M_3(g)$. These expressions do not seem to be related to our blossoming tree representation of
the maps. On the other hand, the existence of a \emph{first order} differential equation for their generating functions 
 seems to be a crucial ingredient: in particular, no simple continued-fraction-like form seems to exist for $M_6(g)$.

\appendix
\section{ A proof of the identity \eqref{drivalue} by verification}
\label{sec:dri}
In order to prove \eqref{drivalue}, we first reformulate the recursion relations \eqref{eq:recurri}, originally expressed in terms of weighted Dyck paths, 
in the equivalent language of sequences. Define 
\begin{equation*}
\mathcal{P}_k=\left\{(u_1,u_2,\ldots,u_k)\in \mathbb{Z}^k\ , u_{j+1}\geq u_j-1\ , \quad 1\leq j\leq k-1\right\}
\end{equation*}
and, if $\mathcal{C}$ and $\mathcal{C'}$ denote comparison relations for integers, denote by $\mathcal{P}_k^{[\mathcal{C},\mathcal{C'}]}$
the subset of $\mathcal{P}_k$ where $u_1$ satisfies the relation $\mathcal{C}$ and $u_k$ the relation $\mathcal{C'}$. For instance
\begin{equation*}
\mathcal{P}_k^{[\geq i,=j]}=\left\{(u_1,u_2,\ldots,u_k)\in \mathcal{P}_k\ , u_{1}\geq i\ \hbox{and}\  u_k=j\right\}\ .
\end{equation*}
To each sequence $u=(u_1,u_2,\ldots,u_k)\in \mathcal{P}_k$, we associate the weight
\begin{equation*}
w(u)=\prod_{j=1}^k\, r_{u_j}(t)
\end{equation*}
where $r_i(t)$, $i\geq 1$, is defined as in Section \ref{sec:orthpol}, hence satisfies \eqref{eq:recurri}, while we set $r_i(t)=0$ for $i\leq 0$. We finally denote by
\begin{equation*}
P_k^{[\mathcal{C},\mathcal{C'}]}=\sum_{u\in \mathcal{P}_k^{[\mathcal{C},\mathcal{C'}]}} w(u)
\end{equation*}
the partition function for weighted sequences.

\medskip
With these notations, the recursion relations \eqref{eq:recurri} for the $r_i(t)$ may be rewritten as
\begin{equation}
r_i(t)=i+\sum_{k\geq 1} t^k g_{k} P_k^{[\geq i, \leq i]}\ , \quad i \geq 1\ .
\label{eq:recurribis}
\end{equation}
\begin{figure}
\begin{center}
\includegraphics[width=8cm]{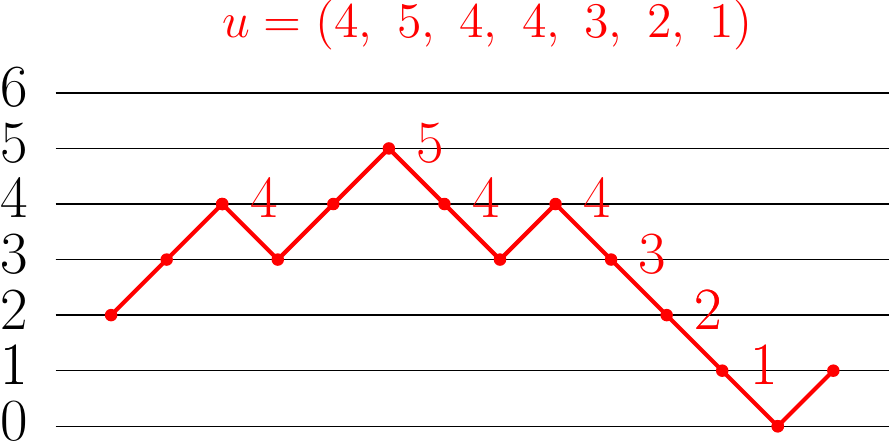}
\end{center}
\caption{A Dyck path of length $2k-1$ (here $k=7$) from height $i$ (here $i=2$) to height $i-1$ and its coding by a sequence $u$ of $\mathcal{P}_k^{[\geq i,\leq i]}$.}
\label{fig:Dycktoseq}
\end{figure}Indeed, a Dyck path in $\mathcal{P}_k^{(i)}$ (the set of Dyck paths with length $2k-1$ from height $i$ to height $i-1$) has $k$ descending steps 
$u_j\to u_{j}-1$, $j=1,\ldots, k$ and is entirely encoded by the sequence $u=(u_1,u_2,\ldots,u_k)$ of these descending steps (see Figure~\ref{fig:Dycktoseq}). 
Clearly this sequence satisfies $u_{j+1}\geq u_j-1$ for all $j$ hence belongs to $\mathcal{P}_k,$ while $u_1\geq i$ (since the Dyck path 
starts at $i$), and $u_k-1 \leq i-1$, hence $u_k\leq i$ (since the Dyck path ends at $i-1$). The sequence is therefore 
an element of $\mathcal{P}_k^{[\geq i, \leq i]}$ and the weight $w(u)$ of a sequence precisely reproduces the desired weight $r_h(t)$
for each  descending step of the Dyck path in  \eqref{eq:recurri}. More precisely, the above encoding provides a bijection between the desired Dyck paths and 
the sequences in $\mathcal{P}_k^{[\geq i, \leq i]}$ where all the elements $u_m$ are positive. This latter positivity constraint can be ignored by noting that configurations where one $u_m$ is negative or zero automatically contribute $0$ to $P_k^{[\geq i, \leq i]}$ since $r_{u_m}(t)=0$.
    
\medskip
Introducing the notation
\begin{equation*}
d_i(t):=2t \frac{d}{dt}\, r_i(t)
\end{equation*}
and differentiating \eqref{eq:recurribis} with respect to $t$ yields the relations
\begin{equation}
d_i(t)=\sum_{k\geq 1}2k\, t^k g_{k} P_k^{[\geq i, \leq i]}+\sum_{k\geq 1}t^k g_{k} P_k^{\bullet[\geq i, \leq i]}\ , \quad i\geq 1\ .
\label{eq:sumP}
\end{equation}
Here $P_k^{\bullet[\geq i, \leq i]}=2t \frac{d}{dt}P_k^{[\geq i, \leq i]}$ is the partition function of sequences $u$ in
$\mathcal{P}_k^{[\geq i, \leq i]}$ \emph{with a marked element $u_m$} and a modified weight
$\prod\limits_{j=1\atop j\neq m}^k\, r_{u_j}(t)\times d_{u_m}(t)$ where the original weight $r_{u_m}(t)$ 
for the marked element was replaced by $d_{u_m}(t)$. We will denote by $u^\bullet$ such marked sequences.
Let us now show that the above equation \eqref{eq:sumP} is satisfied if we set
\begin{equation}
d_i(t)=r_i(t)\big(r_{i+1}(t)+r_{i-1}(t)-2\big)\ , \quad i\in \mathbb{Z}\ .
\label{eq:dival}
\end{equation}
Note in particular that $d_i(t)=0$ for $i\leq 0$.
Inserting the expression \eqref{eq:dival} in \eqref{eq:sumP} yields a right hand side equal to
\begin{equation}
\begin{split}
&\sum_{k\geq 1}2k\, t^k g_{k} P_k^{[\geq i, \leq i]}+\sum_{k\geq 1}t^k g_{k} \Big(P_k^{\uparrow[\geq i, \leq i]}-P_k^{\downarrow[\geq i, \leq i]}-2k\, P_k^{[\geq i, \leq i]}\Big) \\
=&\sum_{k\geq 1}t^k g_{k} \Big(P_k^{\uparrow[\geq i, \leq i]}-P_k^{\downarrow[\geq i, \leq i]}\Big)
\end{split}
\label{eq:PP}
\end{equation}
where $P_k^{\uparrow[\geq i, \leq i]}$ (resp. $P_k^{\downarrow[\geq i, \leq i]}$) enumerates marked sequences $u^\bullet$ 
with a modified weight $r_{u_m+1}(t)r_{u_m}(t)$ (resp. $r_{u_m}(t)r_{u_m-1}(t)$) for the marked element $u_m$. The corresponding total weight for 
the whole sequence $u^\bullet$ will be denoted by $w^\uparrow(u^\bullet)$ (resp. $w^\downarrow(u^\bullet)$), which clearly satisfies
$w^\uparrow(u^\bullet)=r_{u_m+1}(t)w(u)$ (resp. $w^\downarrow(u^\bullet)=r_{u_m-1}(t)w(u)$).

\medskip
We may now easily design a bijection $\Phi$ from $\mathcal{P}_k^\bullet$ to $\mathcal{P}_k^\bullet$ (the set of marked
sequences $u^\bullet$ with $u$ in $\mathcal{P}_k$) such that 
\begin{equation}
w^\uparrow\big(\Phi(u^\bullet)\big)=w^\downarrow(u^\bullet)\ .
\label{eq:eqw}
\end{equation}
\begin{figure}
\begin{center}
\includegraphics[width=8cm]{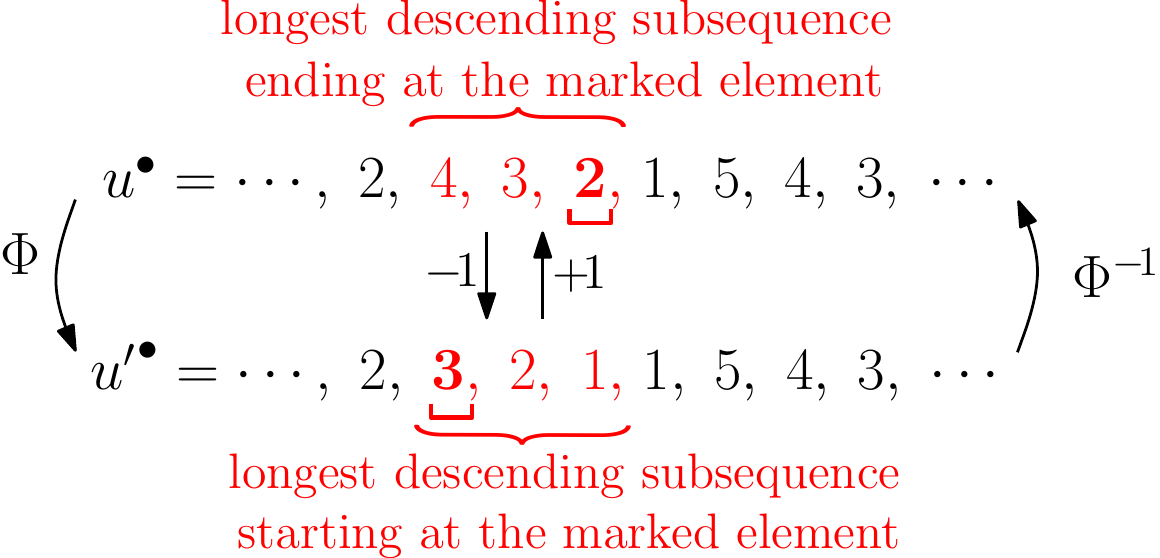}
\end{center}
\caption{An example of the action of the bijection $\Phi$ on a sequence $u^\bullet$ of $\mathcal{P}_k^\bullet$. The resulting
sequence ${u'}^\bullet$ is obtained by shifting by $-1$ the values of the elements of $u$ within the longest descending subsequence 
ending at the marked element (here underlined) and moving the marking at the first element of this modified subsequence, which
is then the the longest descending subsequence 
starting at the marked element in  ${u'}^\bullet$.}
\label{fig:phi}
\end{figure}The bijection is as follows (see Figure~\ref{fig:phi}): to obtain the sequence $u'_m$ associated with $\Phi(u^\bullet)$, we consider the \emph{longest} descending subsequence 
$(u_{m-\ell},u_{m-\ell+1},\ldots, u_m)=(u_m+\ell, u_m+\ell-1, \cdots, u_m)$ ending at the marked element $u_m$ in $u^\bullet$
and replace this subsequence by $({u'}_{m-\ell},\allowbreak{u'}_{m-\ell+1},\ldots, {u'}_m)=({u'}_m+\ell, {u'}_m+\ell-1, \cdots, {u'}_m)$ with ${u'}_m=u_m-1$, keeping
all the other elements unchanged. The new sequence ${u'}^\bullet= \Phi(u^\bullet)$ is now marked at the element ${u'}_{m-\ell}$. In other
words, we obtain $\Phi(u^\bullet)$ from $u^\bullet$ by shifting by $-1$ the elements of the longest descending subsequence ending at the marked
element in $u^\bullet$ and moving the marking at the first element of this subsequence. The sequence $\Phi(u^\bullet)$ is in 
$\mathcal{P}_k^\bullet$ since $u_{m+1}\geq u_m-1$, hence ${u'}_{m+1}=u_{m+1}\geq u_m-1={u'}_m$ and a fortiori ${u'}_{m+1} \geq {u'}_m-1$ 
while $u_{m-\ell-1}\leq u_{m-\ell}=u_{m}+\ell$
(since we chose the longest descending subsequence) hence ${u'}_{m-\ell}=u_{m}+\ell-1\geq u_{m-\ell-1}-1={u'}_{m-\ell-1}-1$.
Clearly, since ${u'}_{m+1}\geq {u'}_m$, the subsequence $({u'}_{m-\ell},{u'}_{m-\ell+1},\ldots, {u'}_m)=({u'}_m+\ell, {u'}_m+\ell-1, \cdots, {u'}_m)$ is the longest descending subsequence starting at the marked element ${u'}_{m-\ell}$ in 
$\Phi(u^\bullet)$ and $\Phi$ is therefore a bijection whose inverse consists in shifting by $+1$ the elements of the longest descending subsequence starting at
its marked element and moving the marking at the end of this subsequence. 
As for the weight $w^\uparrow\big(\Phi(u^\bullet)\big)$, the contribution of the modified subsequence is 
\begin{equation*}
\begin{split}
\big(r_{{u'}_m+\ell+1}(t) r_{{u'}_m+\ell}(t)\big)\prod_{j=0}^{\ell-1} r_{{u'}_m+j}(t)&=
\big(r_{u_m+\ell}(t) r_{u_m+\ell-1}(t)\big)\prod_{j=0}^{\ell-1} r_{u_m+j-1}(t)\\
&=\prod_{j=1}^{\ell} r_{u_m+j}(t)\big(r_{u_m}(t) r_{u_m-1}(t)\big)
\end{split}
\end{equation*}
which matches precisely the contribution of the original subsequence in the weight $w^\downarrow(u^\bullet)$, hence \eqref{eq:eqw}.

If we now restrict the set of marked sequences $u^\bullet$ to the subset $\mathcal{P}_k^{\bullet[\geq i,\leq i]}$ of $\mathcal{P}_k^\bullet$, the image of this subset
by $\Phi$
contains sequences which are not necessarily in $\mathcal{P}_k^{\bullet[\geq i,\leq i]}$. This occurs if (and only if) the sequence $u^\bullet$ starts
with $u_1=i$ and is marked at an element $u_m$ such that $u_1$ is part of its preceding longest descending subsequence. Then
the marked element of ${u'}^\bullet$ is ${u'}_1=i-1$ so that ${u'}^\bullet$ is no longer in $\mathcal{P}_k^{\bullet[\geq i,\leq i]}$. 
Otherwise stated, we have
\begin{equation*}
\begin{split}
\mathcal{A}_k^{(i)}&:=\Phi\big(\mathcal{P}_k^{\bullet[\geq i,\leq i]}\big)\setminus \mathcal{P}_k^{\bullet[\geq i,\leq i]}
\\&=\left\{{u'}^\bullet \in \mathcal{P}_k^\bullet  \ , {u'}_{1}=i-1\ , \  {u'}_k\leq i\  \hbox{and}\ {u'}_1\ \hbox{is the marked element}
\right\}\ .
\end{split}
\end{equation*}
Any sequence in $\mathcal{A}_k^{(i)}$ has a weight $w^\uparrow\big({u'}^\bullet\big)=r_i(t) w(u')$ in terms of the weight $w(u')$ of
the corresponding unmarked sequence $u'$, which is an arbitrary sequence of  $\mathcal{P}_k^{[=i-1,\leq i]}$. We deduce that 
the contribution of the pre-image by $\Phi$ of these sequences to $P_k^{\downarrow[\geq i, \leq i]}$ 
(which de facto has no compensation from $P_k^{\uparrow[\geq i, \leq i]}$) is given by 
\begin{equation*}
A_k^{(i)}=r_i(t) \times  P_k^{[=i-1,\leq i]}\ .
\end{equation*} 
A similar argument shows that the pre-image of $\mathcal{P}_k^{\bullet[\geq i,\leq i]}$
by $\Phi$
contains sequences not necessarily in $\mathcal{P}_k^{\bullet[\geq i,\leq i]}$, with 
\begin{equation*}
\begin{split}
\mathcal{B}_k^{(i)}&:=\Phi^{-1}\big(\mathcal{P}_k^{\bullet[\geq i,\leq i]}\big)\setminus \mathcal{P}_k^{\bullet[\geq i,\leq i]}
\\&=\left\{u^\bullet \in \mathcal{P}_k^\bullet  \ , u_1\geq i\ , \  u_k=i+1\  \hbox{and}\ u_k\ \hbox{is the marked element}
\right\}\ .
\end{split}
\end{equation*}
Any sequence in $\mathcal{B}_k^{(i)}$has a weight $w^\downarrow\big(u^\bullet\big)=r_i(t) w(u)$ in terms of the weight $w(u)$ of
the corresponding unmarked sequence $u$, which is an arbitrary sequence of  $\mathcal{P}_k^{[\geq i,=i+1]}$. We again deduce that 
the contribution of the image of these sequences by $\Phi$ to $P_k^{\uparrow[\geq i, \leq i]}$ (with no compensation from $P_k^{\downarrow[\geq i, \leq i]}$) 
is given by 
\begin{equation*}
B_k^{(i)}=r_i(t) \times  P_k^{[\geq i,=i+1]}\ .
\end{equation*}
\begin{figure}
\begin{center}
\includegraphics[width=8cm]{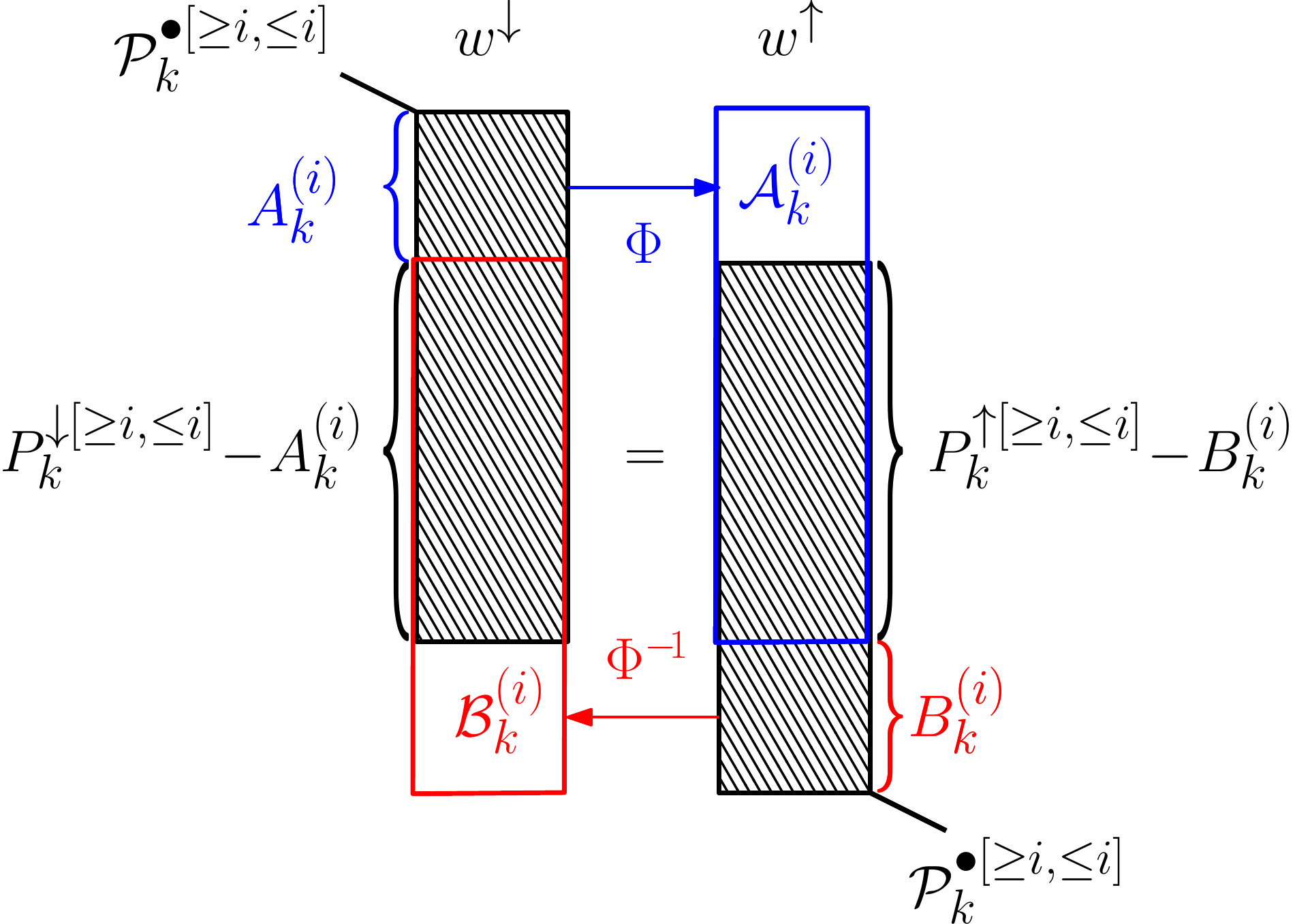}
\end{center}
\caption{A schematic representation of the relation $P_k^{\uparrow[\geq i, \leq i]}-B_k^{(i)}=P_k^{\downarrow[\geq i, \leq i]}-A_k^{(i)}$. Domains facing
each other are image of each other by the bijection $\Phi$ from $\mathcal{P}_k^{\bullet}$ into itself and have the same weight when evaluated 
with $w^\downarrow$ on the left and with $w^\uparrow$ on the right. The shaded region corresponds to the set $\mathcal{P}_k^{\bullet[\geq i,\leq i]}$.}
\label{fig:diffsym}
\end{figure}The resulting relation $P_k^{\uparrow[\geq i, \leq i]}-B_k^{(i)}=P_k^{\downarrow[\geq i, \leq i]}-A_k^{(i)}$ (see Figure~\ref{fig:diffsym}) leads, from \eqref{eq:PP},  to a right hand side in \eqref{eq:sumP} equal to
\begin{equation*}
\begin{split}
&\sum_{k\geq 1}t^k g_{k} \Big(B_k^{(i)}-A_k^{(i)}\Big) \\
=&r_i(t)\times \sum_{k\geq 1}t^k g_{k} \Big(P_k^{[\geq i,=i+1]}-P_k^{[=i-1,\leq i]}\Big)
\end{split}
\end{equation*}
\begin{figure}
\begin{center}
\includegraphics[width=10cm]{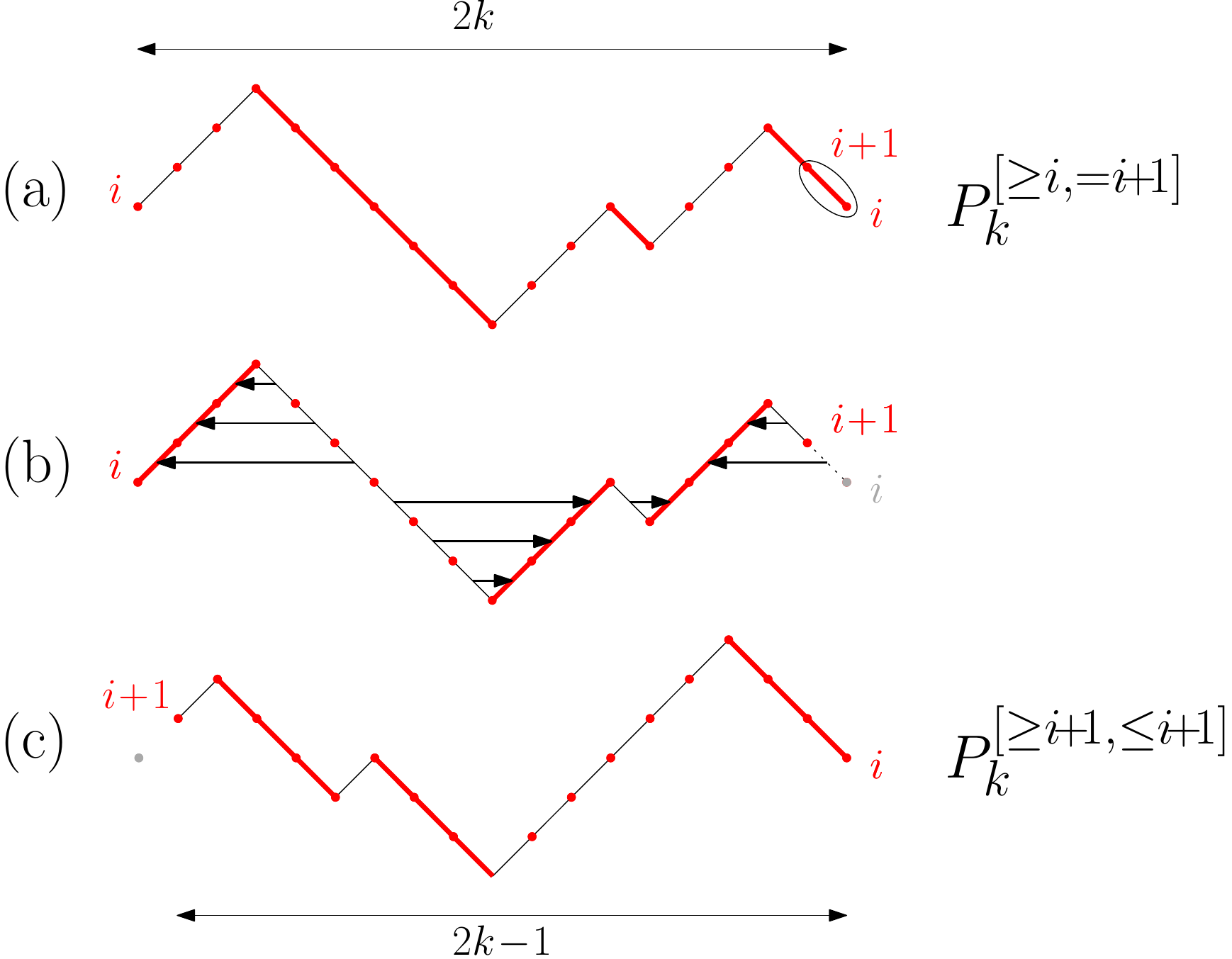}
\end{center}
\caption{A schematic picture of the identification $P_k^{[\geq i,=i+1]}=P_k^{[\geq i+1,\leq i+1]}$ (see text). Steps 
which receive a non-trivial weight $r_h(t)$ for some $h$ are indicated by thick red lines.}
\label{fig:PkPk}
\end{figure}Let us now show the identities
\begin{equation*}
P_k^{[\geq i,=i+1]}=P_k^{[\geq i+1,\leq i+1]}\ , \quad P_k^{[=i-1,\leq i]}=P_k^{[\geq i-1,\leq i-1]}\ .
\end{equation*}
To prove these identities, it is simpler to return to the Dyck path interpretation of the involved generating functions.
For instance,  from the sequence encoding of Dyck paths, $P_k^{[\geq i,=i+1]}$ is the generating function of Dyck 
paths\footnote{Again a sequence in $\mathcal{P}_k^{[\geq i,=i+1]}$ encoding for a path with negative 
heights automatically contributes $0$ to $P_k^{[\geq i,=i+1]}$ since $r_h(t)=0$ for $h\leq 0$.} starting at the height $i$ and ending with a down step $i+1\to i$, with a weight $r_h(t)$ per descending step $h\to h-1$, and with $k$ descending steps, 
hence a total length $2k$ (see Figure~\ref{fig:PkPk}(a)). 
These paths from height $i$ to height $i$
have the same number of ascending steps $h-1\to h$ as that of descending steps  $h\to h-1$ \emph{for each value of $h$}.
The quantity $P_k^{[\geq i,=i+1]}$ is therefore also the generating function of Dyck paths of length $2k$, 
starting at the height $i$ and ending with a down step $i+1\to i$, with a weight $r_h(t)$ per ascending step $h-1\to h$
or, by removing the last step (now with weight $1$ since it is descending),  
as the generating function of Dyck paths of length $2k-1$
starting at the height $i$ and ending at height $i+1$, with a weight $r_h(t)$ per ascending step $h-1\to h$ (see Figure~\ref{fig:PkPk}(b)). If we now reverse
these paths (from left to right),  $P_k^{[\geq i,=i+1]}$ is the generating function of Dyck paths 
of length $2k-1$ starting at the height $i+1$ and ending at height $i$, with a weight $r_h(t)$ per descending step $h\to h-1$, which is precisely 
the Dyck path interpretation of $P_k^{[\geq i+1,\leq i+1]}$ (see Figure~\ref{fig:PkPk}(c)). This proves the first identity above, while, as may easily be checked
by the reader, 
the second identity is proved along fully similar lines.

\medskip
The right hand side in \eqref{eq:sumP} eventually reads
\begin{equation*}
\begin{split}
&r_i(t)\times \sum_{k\geq 1}t^k g_{k} \Big(P_k^{[\geq i+1,\leq i+1]}-P_k^{[\geq i-1,\leq i-1]}\Big)\\
&= r_i(t)\Big(r_{i+1}(t)-(i+1)-\big(r_{i-1}(t)-(i-1)\big)\Big)\\
&=r_i(t)\big(r_{i+1}(t)-r_{i-1}(t)-2\big)
\end{split}
\end{equation*}
where we used \eqref{eq:recurribis} for $i\to i+1$ and $i\to i-1$, with $i\geq 1$ (note that the relation \eqref{eq:recurribis}, valid
for $i\geq 1$ may in practice be extended to $i=0$ as it yields $0=0$). This precisely matches the expected value $d_i(t)$ for the left hand side of 
\eqref{eq:sumP} with our Ansatz \eqref{eq:dival}. The expression \eqref{eq:dival} for $d_i(t)$ is therefore a solution
of the system \eqref{eq:sumP}, and is a power series in the variable~$t$. Since \eqref{eq:sumP} determines $d_i(t)$ uniquely
(order by order) as a formal power series in $t$, we deduce that the desired relation \eqref{eq:dival} is indeed satisfied. This is clearly equivalent to \eqref{drivalue}.

\section{A combinatorial proof of the identity \eqref{drivalue}}
\label{sec:driproof}
For $i\geq 1$, we let $\cR_i'$ be the family of maps in $\cR_{i}$ with a secondary marked half-edge $h$; note that the counting series of $\cR_i'$ is $2t\frac{d}{dt}r_i(t)$.  
We define $\cF_i$ as the subfamily of maps $M$ in $\cR_i$ where the root vertex has degree~$4$, and such that if we let $e_0,e_1,e_2,e_3$ be the edges containing the $4$ incident half-edges in clockwise order around the root vertex $v'$ (starting from the root corner), then $e_1$ is not a loop nor marked, and $M$ admits a compatible Eulerian orientation where $e_0$ and $e_1$ are going out of $v'$. 

\begin{figure}
\begin{center}
\includegraphics[width=3cm]{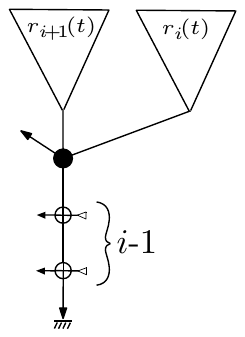}
\end{center}
\caption{The configuration of an $i$-enriched Eulerian tree corresponding to a map in $\overline{\cF}_i$ (with its root leaf extended into
a branch of length $i$).}
\label{fig:tree_riprime}
\end{figure}

\begin{lem}\label{lem:counting_fi}
The counting series of $\cF_i$ is $t^2g_2 r_i(t)(r_{i+1}(t)-r_{i-1}(t)-2)$.   
\end{lem} 
\begin{proof}
The counting series of $i$-enriched Eulerian trees where the root node has degree $4$ is $t^2 g_2r_i(g)(r_{i+1}(g)+r_i(g)+r_{i-1}(g))$. The root node has $3$ children one of which is an opening leaf $\ell$, and the $3$ terms correspond to $\ell$ being the left child, middle child, and right child respectively. By Lemma~\ref{lem:outroot}, $\ell$ is the left child iff the corresponding map in $\cR_i$ (whose root vertex has degree $4$) has a compatible root-accessible Eulerian orientation such that the edges $e_0$ and $e_1$ are outgoing. Hence, if we let $\overline{\cF}_i$ be the family defined as $\cF_i$ but allowing $e_1$ to be a loop or to be marked, then the counting series of $\overline{\cF}_i$ is   
$t^2g_2 r_{i+1}(t)r_{i}(t)$, see Figure~\ref{fig:tree_riprime}. The contribution where $e_1$ is marked 
corresponds to having $\ell$ unmatched in the first step of the construction of Section~\ref{sec:interpr_ri} (2nd drawing in Figure~\ref{fig:ri_map}). This amounts to decreasing~$i$ by~$1$ (we can replace $\ell$ by the first artificial opening leaf along the extended branch), hence the counting series for those maps is $t^2g_2 r_i(t)r_{i-1}(t)$. Finally, if $e_1$ is an unmarked loop, then it means that $\ell$ is matched with a closing leaf that is adjacent to the root node. This closing leaf can be the middle or right child, and in both cases the counting series is easily seen to be $t^2g_2r_i(t)$. By subtraction we conclude that the counting series of $\cF_i$ is $t^2g_2 r_i(t)(r_{i+1}(t)-r_{i-1}(t)-2)$. 
\end{proof}

\begin{figure}
\begin{center}
\includegraphics[width=14cm]{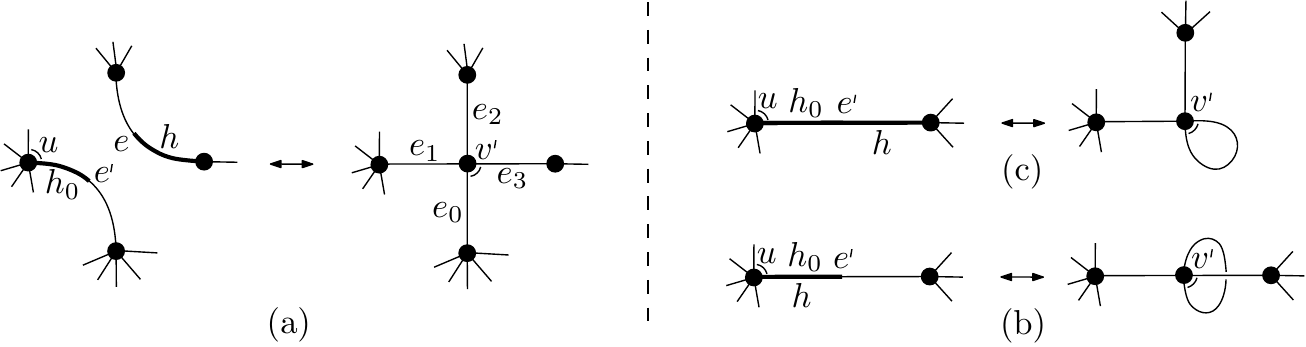}
\end{center}
\caption{The rules (in each type $\mathfrak{a}$, $\mathfrak{b}$ or $\mathfrak{c}$) in the bijection from $\cR_i'$ to $\cF_{i}$.}
\label{fig:rifi}
\end{figure}

We now describe a bijection between $\cR_i'$ and $\cF_i$ such that from $\cR_i'$ to $\cF_i$ the number of edges is increased by $2$, the number of vertices of degree $4$ is increased by~$1$ and the number of vertices of degree $2k$ is preserved for $k\neq 2$ (given Lemma~\ref{lem:counting_fi}, this will give a bijective proof of~\eqref{drivalue}). Moreover the number of marked edges and their multiplicities are also preserved. We will distinguish three types ($\mathfrak{a},\mathfrak{b},\mathfrak{c}$) in the two families, and describe a bijection in each type. 

A map in $\cR_i'$ is said to be of type $\mathfrak{a}$ if the secondary marked half-edge $h$ is not on the same edge as the root half-edge $h_0$, of type $\mathfrak{b}$ if $h=h_0$, and type $\mathfrak{c}$ if $h$ is opposite to $h_0$ on the root edge. 
A map in $\cF_i$ is said to be of type $\mathfrak{a}$ if there is no loop at $v'$. Otherwise if there is a loop at $v'$ it means that either $e_0=e_2$ or $e_0=e_3$ (we can not have $e_2=e_3$, since this would make impossible to have an Eulerian orientation with both $e_0$ and $e_1$ outgoing). In the first (resp. second) case, the map is said to be of type $\mathfrak{b}$ (resp. of type $\mathfrak{c}$).   

\begin{figure}
\begin{center}
\includegraphics[width=14cm]{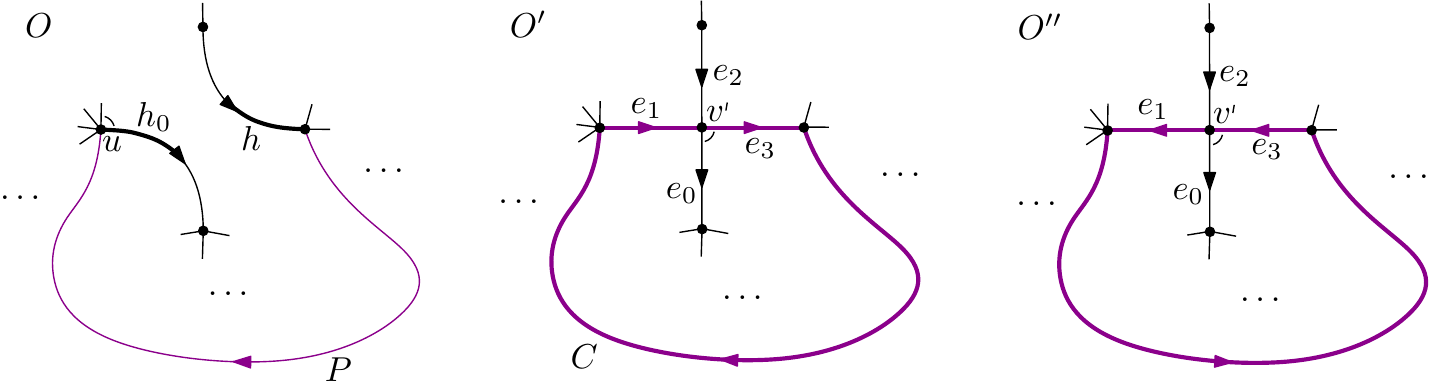}
\end{center}
\caption{In type $\mathfrak{a}$, the existence of a path from the vertex at $h$ to the root vertex guarantees that $M'$ is in $\cF_i$.}
\label{fig:reverse}
\end{figure}

In type $\mathfrak{a}$, the bijection from $\cR_i'$ to $\cF_i$ is as shown in Figure~\ref{fig:rifi}(a). 
Let $M\in\cR_i'$, and let $e'$ be the root edge and $e$ the edge of $h$. We join $e'$ and $e$ at their respective middles, so as to create a new vertex $v'$ of degree $4$ taken as the new root vertex, and so that the edges $(e_0,e_1,e_2,e_3)$ around $v'$ respectively arise from $(\mathrm{opp}(h_0),h_0,\mathrm{opp}(h),h)$. If $e'$ is marked, we mark $e_0$ (with same multiplicity as $e'$) and orient it out of $v'$ (and keep $e_1$ unmarked). If $e$ is marked then we mark exactly one of $\{e_2,e_3\}$ (with same  multiplicity as $e$) and orient it toward $v'$: we mark $e_2$ if $h$ is ingoing and mark $e_3$ if $h$ is outgoing. 
The obtained map $M'$ has root vertex $v'$ of degree $4$ and no loop at $v'$. It has the same number of marked edges as in $M$, and with same multiplicities.

We now show that $M'$ is in $\cF_i$, i.e., that it admits a compatible root-accessible Eulerian orientation such that $e_0$ and $e_1$ are going out of $v'$. By definition of $\cR_i$, $M$ admits a compatible root-accessible Eulerian orientation $O$ such that $e'$ is going out of the root vertex $u$ of $M$. 
Let $O'$ be the induced compatible Eulerian orientation of $M'$ (at $v'$, $e_0$ is outgoing, $e_1$ is ingoing, and exactly one of $\{e_2,e_3\}$ is outgoing, which is $e_3$ if $h$ is ingoing and $e_2$ if $h$ is outgoing). Clearly $O'$ is root-accessible. It remains to show that $M'$ admits a compatible Eulerian orientation where both $e_0$ and $e_1$ are outgoing. Since $O$ is root-accessible, in $M$ there is an oriented path $P$ of unmarked edges from the end of $e$ to
the root vertex $u$. In $O'$ this path extends into a directed cycle $C$ of unmarked edges passing at $v'$ by $e_1$ and by one of $\{e_2,e_3\}$, see Figure~\ref{fig:reverse}. Reversing this cycle we obtain a compatible Eulerian orientation $O''$ of $M'$ (which is also root-accessible) where both $e_0,e_1$ are outgoing at $v'$. Hence $M'$ is in $\cF_i$
of type $\mathfrak{a}$.

The inverse mapping, from a map $M'\in\cF_i$, applies the operation of Figure~\ref{fig:rifi}(a) in the reverse direction.  
Concerning marked edges, if $e_0$ is marked it is going out of $v'$, then we mark the new root edge $e'$ (with same multiplicity as $e_0$) and orient it out of $u$. If $e_2$ or $e_3$ is marked (not both can be marked otherwise any Eulerian orientation where $e_0,e_1$ are outgoing would not be root-accessible) it has to be ingoing at $v'$ (since there is an Eulerian orientation with $e_0,e_1$ outgoing at $v'$), then we mark $e$ (with same multiplicity) and orient it with the same direction. We have to prove that the obtained map $M$ is admissible. By definition $M'$ admits a compatible root-accessible Eulerian orientation $O''$ where $e_0,e_1$ are going out of $v'$. By accessibility there exists an oriented path $P'$ of unmarked edges from $u$ (the extremity $\neq v'$ of $e_1$) to $v'$; note that the last edge of $P$ is $e_2$ or $e_3$. Let $C$ be the directed cycle made of $P'$ plus the edge $e_1$. Returning $C$ we obtain a new compatible Eulerian orientation $O'$.  
 This orientation clearly induces a compatible Eulerian orientation $O$ of $M$, where the root edge $e'$ is going out of $u$. Moreover the presence of the directed cycle $C$ in $O'$ easily ensures that $O$ is root-accessible. Hence $M$ is in $\cR_i'$ of type $\mathfrak{a}$.

In type $\mathfrak{b}$ the bijection is as shown in Figure~\ref{fig:rifi}(b): we create a new vertex $v'$ in the middle of $e'$ and attach a new loop at it, so that $v'$ is the new root vertex and the loop-edge is $e_0=e_2$. The edges $e_1,e_3$ remain unmarked, and if $e'$ is marked then we mark the loop-edge (with same multiplicity) and orient it so that the new root half-edge is outgoing. The obtained map $M'$ clearly admits a compatible root-accessible Eulerian orientation $O'$ with $e_0$ outgoing and $e_1$ ingoing at $v'$. By definition $M$ admits a compatible root-accessible Eulerian orientation where $e'$ is outgoing at the root vertex $u$ of $M$, which induces a compatible Eulerian orientation of $M'$ with $e_1$ ingoing (and $e_3$ outgoing) at $v'$. As before the existence in $O$ of a directed path of unmarked edges from the end of $e'$ to its origin ensures that in $O'$ there is a directed cycle passing at $v'$ by $e_1$ and $e_3$, and reversing this cycle yields a compatible Eulerian orientation $O''$ of $M'$ with $e_1$ (and the root half-edge) outgoing at $v'$, so that $M'\in\cF_i$. The inverse bijection, from a map $M'\in\cF_i$ of type $\mathfrak{b}$, applies the reverse operation of Figure~\ref{fig:rifi}(b), and the verification that one obtains a map in $\cR_i'$ relies again on similar arguments as in type $\mathfrak{a}$. Finally, in type $\mathfrak{c}$ the bijection is completely analogous to type $\mathfrak{b}$ up to exchanging the roles of $e_2$ and $e_3$, see Figure~\ref{fig:rifi}(c).

\section{Derivation of \eqref{eq:sysrq} from bi-orthogonal polynomials}
\label{sec:mregularint}
As in Section \eqref{sec:intrep}, we may obtain an integral representation for the generating function of $m$-regular bipartite maps with a weight $g$
per black vertex. Here we start by setting, for $p,q\in \mathbb N$
\begin{equation}
\int dx\, dy\, {\rm e}^{-x\, y}\, x^p\, y^q := p!\,  \delta_{p,q} 
\label{eq:defint}
\end{equation}
which, by linearity, defines an integral $\int dx dy\, F(x,y)$ for any polynomial in the variables $x$ and $y$. This definition is moreover compatible with the
notion of integration by parts, as it allows to set
\begin{equation*}
\begin{split}
\int dx\, dy\frac{\partial}{\partial y}\left(\, {\rm e}^{-x\, y}\, y^q\right)\, x^p\, y^s &=
\int dx\, dy\, {\rm e}^{-x\, y}\, (-x\, y^q+q\, y^{q-1})x^p\, y^s\\
&=-(p+1)!\, \delta_{p+1,q+s}+q\, p!\, \delta_{p,q+s-1}\\
&=-(p+1-q)\, p!\, \delta_{p,q+s-1}=- s\, p!\delta_{p,q+s-1}\\
&=-\int dx\, dy\, {\rm e}^{-x\, y}\, y^q x^p\, (s\, y^{s-1})\\
&=-\int dx\, dy\, {\rm e}^{-x\, y}\, y^q x^p\, \frac{\partial}{\partial y}y^s\ .\\
\end{split}
\end{equation*}
We may then consider the integral  
\begin{equation}
h_0(g)=\int dx\, dy\,  {\rm e}^{-x\, y}\, {\rm e}^{V(g,x,y)}\ ,
\quad 
V(g,x,y)=\sqrt{g}\left(\frac{x^{m}}{m}+\frac{y^{m}}{m}\right)
\label{eq:defh0c}
\end{equation}
understood as a formal power series in $g$. Note that, since the integral of a monomial in $x$ and $y$ is non-zero only if the degree in $x$
matches that in $y$, only terms with integer powers in $g$ survive in the expansion in $\sqrt{g}$. {}From the definition \eqref{eq:defint}, where $p!\,  \delta_{p,q}$ 
corresponds to the number of pairings between a set of $p$ objects and a set of $q$ objects, $h_0(g)$ 
may be identified in terms of maps as $h_0(g)=1+\sum_{V_\bullet\geq 1}g^{V_\bullet} Z_{V_\bullet}$ where $Z_{V_\bullet}$ denotes the generating 
function for possibly disconnected $m$-regular bipartite maps
with a total of $V_\bullet$ black vertices, and with suitable symmetry factors. As before, we may instead consider the associated
generating functions $W_{V_\bullet}$ for connected rooted $m$-regular bipartite maps with $V_\bullet$ black vertices, with power series
\begin{equation}
\sum_{V_\bullet\geq 1}g^{V_\bullet} W_{V_\bullet}=m\, g\frac{d}{dg}{\rm Log}\,  h_0(g)\ .
\label{eq:genrm}
\end{equation}  
To get a recursive system determining this counting series, we may introduce bi-orthogonal polynomials 
$p_i:=p_i(g,x)$, $i\in \mathbb{N}$ and $\tilde{p}_j:=\tilde{p}_j(g,y)$, $j\in \mathbb{N}$ satisfying
\begin{equation*}
\langle p_i | \tilde{p}_j \rangle =h_i(g) \, \delta_{i,j}\ ,  \qquad p_i(g,x)=x^i+\sum_{k<i} a_{i,k}(g) x^k, \ \tilde{p}_j(g,y)=y^j+\sum_{k<j} \tilde{a}_{j,k}(g) y^k
\end{equation*}
with respect to the ``scalar product''
\begin{equation*}
\langle F | G \rangle:=\int dx\, dy\,  {\rm e}^{-x\, y}\, {\rm e}^{V(g,x,y)}\  F(g,x)\, G(g,y)\ .
\end{equation*}
Note that $h_0(g)$ matches its definition in \eqref{eq:defh0c} since $p_0(g,x)=1$ and $\tilde{p}_0(g,y)=1$. {}From the obvious $x\leftrightarrow y$ 
symmetry, we also have $\tilde{p}_j(g,y)=p_j(g,y)$ for all $j$ so we have to deal in practice with a single family of polynomials. We also have the property
that $\langle x^i| y^j\rangle=0$ if $i\neq j\mod[m]$ (since $V(g,x,y)$ is a polynomial in $x^m$ and $y^m$), from which we deduce that $p_i(g,x)$ contains only powers $x^k$ of $x$ such that 
$k=i\mod[m]$. We may then,
as was done in Section~\ref{sec:orthpol}, write the decomposition
\begin{equation}
x\, p_i(g,x)=p_{i+1}(g,x)+\frac{1}{\sqrt{g}}q_{i-m+2}(g)\, p_{i-m+1}(g,x)
\label{eq:xpim}
\end{equation}  
(the choice of the index for $q_i$ and the normalization by $\frac{1}{\sqrt{g}}$ are for future convenience) where $q_i(g)=0$ for $i\leq 0$. 
The absence in the right hand side of terms proportional to $p_{i-k\, m+1}(g,x)$ with $k > 1$ is a consequence of the
identity 
\begin{equation*}
\!\!\!\! \frac{\partial}{\partial y} w(g,x,y)+\left(x-\sqrt{g}\, y^{m-1}\right) w(g,x,y)=0\ ,
\quad w(g,x,y):= {\rm e}^{-x\, y+V(g,x,y)}
\end{equation*}
which, for $j\leq i+1$, implies the relation
\begin{equation*}
\begin{split}
\!\!\!\!\!\!\!\!\!\!\!\!\!\!\!\!\!\langle x p_{i} | p_j \rangle&= \int dx\, dy\,  w(g,x,y)\, x\, p_i(g,x)p_j(g,y)\\
&= \int dx\, dy\, \left(-\frac{\partial}{\partial y} w(g,x,y)+\sqrt{g}\, y^{m-1}w(g,x,y)\right)  p_i(g,x)p_j(g,y)\\
&=\langle p_i | \frac{\partial}{\partial y}p_j \rangle+\sqrt{g}\langle p_i |  y^{m-1}\, p_j\rangle\ .
\end{split}
\end{equation*}
This immediately gives $\langle x p_{i} | p_j \rangle=0$ for $j< i+1-m$, hence the absence of terms $p_{i-k\, m+1}(g,x)$ with $k > 1$ in \eqref{eq:xpim}.
For $j=i-m+1$,  it leads to
\begin{equation}
\frac{1}{\sqrt{g}}q_{i-m+2}(g)\, h_{i-m+1}(g)=\sqrt{g}\langle p_i |  y^{m-1}\, p_{i-m+1}\rangle=\sqrt{g}\, h_i(g)
\label{eq:Qmsys}
\end{equation}
while, for $j=i+1$, we get
\begin{equation}
\begin{split}
h_{i+1}(g)&=(i+1)\, h_i(g)+\sqrt{g}\langle p_i |  y^{m-1}\, p_{i+1}\rangle\\
&=h_i(g)\left((i+1)+\sqrt{g}\sum_{a=0}^{m-2}\frac{1}{\sqrt{g}}q_{i+1-a}(g)\right)\\
\end{split}
\label{eq:hmsys}
\end{equation}
where the last equation was obtained by repeated actions of \eqref{eq:xpim}.

\medskip
Defining $r_i(g)=h_i(g)/h_{i-1}(g)$ as before, we obtain the desired system of equations
\begin{equation}
r_i(g)=i+ \sum_{a=0}^{m-2}q_{i-a}(g)\ , \quad
q_i(g)=g\, \prod_{a=0}^{m-2}r_{i+a}(g)\ , \quad i\geq 1\ .
\label{eq:sysQ}
\end{equation}
In the first equation, it is implicitly assumed that $q_j(g)=0$ for $j\leq 0$ and in particular
\begin{equation*}
r_1(g)=1+q_{1}(g)\ .
\end{equation*}
The system \eqref{eq:sysQ} is nothing but \eqref{eq:sysrq}, and determines recursively all the $r_i(g)$ as power series in $g$.

\medskip
Let us finally make the connection between $r_1(g)$ and the desired map generating function \eqref{eq:genrm}.
We have
\begingroup
\allowdisplaybreaks
\begin{align*}
m\, g\frac{d}{dg}{\rm Log}\,  h_0(g)
&= \frac{1}{h_0(g)}
\int dx\, dy\,  {\rm e}^{-x\, y}\, m\, g\, \frac{\partial}{\partial g} {\rm e}^{V(g,x,y)}\\
&= \frac{1}{h_0(g)} \int dx\, dy\,  {\rm e}^{-x\, y}\, \frac{1}{2}\left(x \frac{\partial}{\partial x}
+y\frac{\partial}{\partial y}\right){\rm e}^{V(g,x,y)}
\\
&= \frac{1}{h_0(g)} \int dx\, dy\,  {\rm e}^{-x\, y}\, y \frac{\partial}{\partial y}
{\rm e}^{V(g,x,y)}
\\
&= -\frac{1}{h_0(g)}
\int dx\, dy\,  \frac{\partial}{\partial y}\left( y\, {\rm e}^{-x\, y}\right) {\rm e}^{V(g,x,y)}\\
\\
&= -\frac{1}{h_0(g)}
\int dx\, dy\, \left(1-x\,y\right) w(g,x,y)\\
&= -1+\frac{1}{h_0(g)}\langle x p_0|y p_0\rangle\\
&= -1+\frac{h_1(g)}{h_0(g)}\\
\end{align*}
\endgroup
since $p_1(g,x)=x$ (which is the only power $x^k$ with $k\leq 1$ and $k=1\mod[m]$), hence $\langle x p_0|y p_0\rangle =\langle p_1|p_1\rangle=h_1(g)$.
We end up with 
\begin{equation*}
r_1(g)=1+m\, g\frac{d}{dg}{\rm Log}\,  h_0(g)=1+\sum_{V_\bullet \geq 1}g^{V_\bullet} W_{V_\bullet}
\end{equation*} 
which again identifies the first term $r_1(g)$ of our recursive system \eqref{eq:sysQ} with the desired map generating function.

\section{ A proof of the identity \eqref{eq:drimbis} by verification}
\label{sec:drim}
To prove the relation \eqref{eq:drimbis},
we start with the system \eqref{eq:sysQ} which we differentiate w.r.t. $g$, giving
\begin{equation}
dr_i(g)=\sum_{a=0}^{m-2} dq_{i-a}(g)\ , \qquad 
dq_i(g)=q_i(g)\left(1+ \sum_{a=0}^{m-2} \frac{dr_{i+a}(g)}{r_{i+a}(g)}\right)
\label{eq:sysQder}
\end{equation}
with
\begin{equation*}
dr_i(g):= g\frac{d}{dg}r_i(g)\, \quad dq_i(g):=g\frac{d}{dg}q_i(g)\ .
\end{equation*}
Let us show that \eqref{eq:sysQder} is satisfied if we set
\begin{equation}
dr_i(g)=\frac{1}{m}r_i(g)\big(q_{i+1}(g)-q_{i-m+1}(g)\big)\ .
\label{eq:dimsolu}
\end{equation}
We first compute the resulting value of $dq_i(g)$ from the second equation in \eqref{eq:sysQder}.
Inserting \eqref{eq:dimsolu} in \eqref{eq:sysQder} yields
\begin{equation*}
\begin{split}
dq_i(g)&=q_i(g)\left(1+ \frac{1}{m}\sum_{a=0}^{m-2} \big(q_{i+a+1}(g)-q_{i+a-m+1}(g)\big) \right)\\
&=q_i(g)\left(1+ \frac{1}{m} \Big(r_{i+m-1}(g)-(i+m-1)-\big(r_{i-1}(g)-(i-1)\big)\Big)\right)\\
&=\frac{1}{m}q_i(g)\big( r_{i+m-1}(g)-r_{i-1}(g)\big)\ .\\
\end{split}
\end{equation*}
{}From the expression \eqref{eq:sysQ} of $q_i(g)$, we deduce
\begin{equation*}
dq_i(g)=\frac{g}{m}\big(\pi_{i}(g)-\pi_{i-1}(g)\big)\ , \qquad \pi_i(g)= \prod_{a=0}^{m-1}r_{i+a}(g)\ .
\end{equation*}
Plugging this value in the right hand side of the first equation in \eqref{eq:sysQder}, the corresponding sum is therefore
telescopic, with value
\begin{equation*}
\begin{split}
\frac{g}{m}\big(\pi_{i}(g)-\pi_{i-m+1}(g)\big)&=\frac{g}{m}\, \prod_{a=0}^{m-1}r_{i+a}(g)-\frac{g}{m}\, \prod_{a=0}^{m-1}r_{i+a-m+1}(g)\\
&=\frac{1}{m}\, r_i(g)\big(q_{i+1}(g)-q_{i-m+1}(g)\big)\\
\end{split}
\end{equation*}
which matches precisely the left hand side $dr_i(g)$ for our Ansatz \eqref{eq:dimsolu}. 
Since \eqref{eq:sysQder} determines $dr_i(g)$ and $dq_i(g)$ uniquely as power series in $g$, this proves \eqref{eq:dimsolu}
which is clearly equivalent to \eqref{eq:drimbis}. 

\section{ A combinatorial proof of the identity \eqref{eq:drimbis}}
\label{sec:proofdqim}
We give here a bijective proof of the identity
\begin{equation}\label{eq:dqi}
m\frac{d}{dg}q_i(g)=\pi_i(g)-\pi_{i-1}(g),\ \ \mathrm{where}\ \pi_i(g):=\prod_{a=0}^{m-1}r_{i+a}(g),
\end{equation}
which itself implies~\eqref{eq:drimbis} since 
\[
\!\!\!\!\!\!\!\!\!\!\!\! m\, g\, \frac{d}{dg}r_i(g)=\sum_{a=0}^{m-2}m\, g\, \frac{d}{dg}q_{i-a}(g)=g\big(\pi_i(g)-\pi_{i-m+1}(g)\big)=r_i(g)\big(q_{i+1}(g)-q_{i-m+1}(g)\big).
\]

We let $\cQ_i'$ be the family of maps in $\cQ_i$ with a secondary marked corner at a black vertex. Note that $m\frac{d}{dg} q_i(g)$ is the counting series of $\cQ_i'$ with $g$ conjugate to the number of black vertices minus $1$. For a map $M\in\cQ_i'$, we denote by $c_{\circ},c_{\bullet}$ the root corner and secondary marked corner, and by $v_{\circ},v_{\bullet}$  their respective incident vertices (so $v_{\circ}$ is the root vertex). The edges incident to $v_{\circ}$ (resp. $v_{\bullet}$) in clockwise order starting from $c_{\circ}$ (resp. $c_{\bullet}$) are denoted $e_0,\ldots,e_{m-1}$ (resp. $e_m,\ldots,e_{2m-1}$).  
Let $\tau(M)$ be the rooted map obtained by merging $v_{\circ}$ and $v_{\bullet}$ at their corners $c_{\circ},c_{\bullet}$, thereby creating a new vertex $v'$ of degree $2m$. Note that the edges are preserved by this operation, if any of them is marked it is kept marked with same orientation and multiplicity. As the root corner of $\tau(M)$ we take the one at $v'$ between $e_{2m-1}$ and $e_0$, so that in clockwise order around $v'$ (starting at the root corner) the incident edges are $e_0,\ldots,e_{m-1},e_m,\ldots,e_{2m-1}$, see Figure~\ref{fig:MtotauM}. 

\begin{figure}
\begin{center}
\includegraphics[width=12cm]{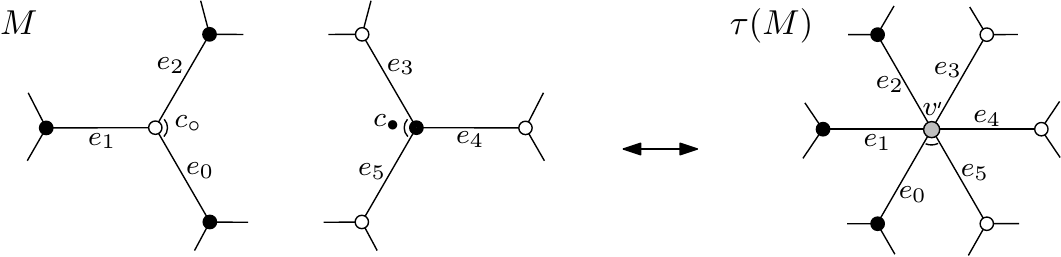}
\end{center}
\caption{The merging operation to obtain $\tau(M)$ from $M\in\cQ_i'$ (case $m=3$ here).}
\label{fig:MtotauM}
\end{figure}

We define a \emph{quasi-$m$-bipartite map} as a rooted map $M$ with a root vertex $v'$ of degree $2m$, which is considered gray, while the other vertices, all of degree $m$, are either black or white, such that there is no white-white edge nor black-black edge, and if we let $e_0,\ldots,e_{2m-1}$ be the edges incident to $v'$ in clockwise order starting from the root corner, then $e_0,\ldots,e_{m-1}$ lead to black vertices while $e_m,\ldots,e_{2m-1}$ lead to white vertices. 
A \emph{1-orientation} of $M$ is an orientation where the white vertices have outdegree $m-1$, the black vertices have outdegree~$1$, and the gray vertex has outdegree $m$. A marked quasi-$m$-bipartite map is called \emph{admissible} if 
it admits a compatible 1-orientation that is root-accessible, and such that the $m$ outgoing edges at the root vertex are $e_0,\ldots,e_{m-1}$, with $e_{m-1}$ unmarked. We let $\cF_i$ be the family of admissible marked quasi-$m$-bipartite maps with multiplicities adding up to less than $i$. 

\begin{figure}
\begin{center}
\includegraphics[width=5cm]{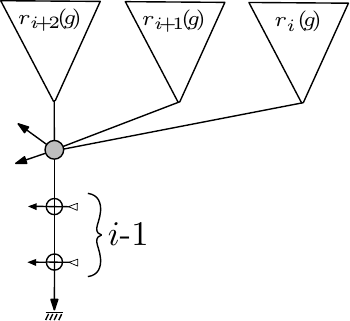}
\end{center}
\caption{The configuration of an $i$-enriched tree corresponding to a map in $\overline{\cF}_i$ (case $m=3$ here).}
\label{fig:tree_qiprime}
\end{figure}

\begin{lem}\label{lem:Qiprime}
Let $M'$ be a marked quasi-$m$-bipartite map with multiplicities adding up to less than $i$. Then $M'$ is of the form $\tau(M)$ for some (necessarily unique) $M\in\cQ_i'$ iff $M'\in\cF_i$. Hence $\tau$ is a bijection from $\cQ_i'$ to $\cF_i$ under which the number of black vertices is decreased by $1$. 
The counting series of $\cF_i$ is $\pi_i(g)-\pi_{i-1}(g)$, with $g$ conjugate to the number of black vertices.  
\end{lem}
\begin{proof}
Let $M\in\cQ_i'$. By definition, $M$ admits a compatible 1-orientation $O$ where the unique ingoing edge at $v_{\circ}$ is $e_{m-1}$. This edge has to be unmarked, otherwise $v_{\circ}$ could not be reached from any other vertex using a directed path of unmarked edges. The orientation $O$ induces a compatible 1-orientation $O'$ of $\tau(M)$, such that $e_{0},\ldots,e_{m-2}$ are outgoing and $e_{m-1}$ is ingoing at the gray vertex $v'$ of $M'$. Clearly this orientation is root-accessible. By accessibility, in $O$ there is an oriented path $P$ of unmarked edges from $v_{\bullet}$ to $v_{\circ}$. In $\tau(M)$ this path becomes a directed cycle $C$ of $O'$ that can be reversed, yielding a compatible (and also root-accessible) 1-orientation $O''$ of $\tau(M)$ such that the $m$ outgoing edges at the gray vertex are $e_0,\ldots,e_{m-1}$. Hence $\tau(M)\in\cF_i$.

Conversely, let $M'\in\cF_i$. By definition $M'$ admits a compatible root-accessible 1-orientation $O''$ such that the $m$ outgoing edges at the root vertex $v'$ are $e_0,\ldots,e_{m-1}$, and $e_{m-1}$ is unmarked. Letting $u$ be the extremity of $e_{m-1}$, by accessibility there exists a directed path $P'$ of unmarked edges from $u$ to $v'$. Hence $e_{m-1}\cup P'$ forms a directed cycle $C$ of unmarked edges. Reversing the orientation of $C$ and splitting the root vertex $v'$ using the operation in Figure~\ref{fig:MtotauM} from right to left, we obtain a marked $m$-regular bipartite map $M$ with multiplicities adding up to less than~$i$, and  endowed with a compatible 1-orientation $O$ such that the unique ingoing edge at the root vertex $v_{\circ}$ is $e_{m-1}$. In addition  $O$ is root-accessible (indeed the presence of the directed cycle $C$ in $O''$ ensures that, in $O$, $v_{\circ}$ can be accessed from $v_{\bullet}$ by a directed path of unmarked edges). Hence $M\in\cQ_i'$ and $M'=\tau(M)$.  
 
We now show that the counting series of $\cF_i$ is $\pi_i(g)-\pi_{i-1}(g)$. We let $\overline{\cF}_i$ be the family defined as $\cF_i$ but allowing $e_{m-1}$ to be marked. Let $M'\in\overline{\cF}_i$, and let $O_{\mathrm{min}}$ be the canonical 1-orientation of $M'$, and $T$ its canonical spanning tree. By Lemma~\ref{lem:outroot} the $m$ outgoing edges at the root vertex $v'$ of $M'$ are $e_0,\ldots,e_{m-1}$, hence are external edges. We obtain the corresponding marked blossoming tree (with
multiplicities adding up to less than $i$) by cutting the external edges at their middles, and transform this marked tree into an $i$-enriched tree  of the form shown in Figure~\ref{fig:tree_qiprime} by using a construction similar to that of Section~\ref{sec:interpr_ri}. Thus the counting series of $\overline{\cF}_i$ is $\pi_i(g)$, with $g$ conjugate to the number of black vertices (black nodes in the blossoming tree). Similarly as in Lemma~\ref{lem:counting_fi}, requiring that $e_{m-1}$ is marked is the same as requiring that the opening leaf arising from $e_{m-1}$ is not matched to a non-artificial closing leaf.  This amounts to decreasing $i$ by $1$ (we can replace the
leaf arising from $e_{m-1}$ by the first artificial opening leaf along the extended branch). Hence the counting series for maps in $\overline{\cF}_i$ where $e_{m-1}$ is marked is $\pi_{i-1}(g)$. By subtraction we conclude that the counting series of $\cF_i$ is $\pi_i(g)-\pi_{i-1}(g)$, with $g$ conjugate to the number of black vertices. 
\end{proof}

Lemma~\ref{lem:Qiprime} then directly yields the identity~\eqref{eq:dqi}, since it ensures that the counting series of $\cQ_i'$, which is 
$m\frac{d}{dg}q_i(g)$ (with $g$ conjugate to the number of black vertices minus $1$) is equal to the counting series of $\cF_i$, which is $\pi_i(g)-\pi_{i-1}(g)$ (with $g$ conjugate to the number of black vertices).

\section*{Acknowledgements} 
We thank S\'everin Charbonnier and Philippe Di Francesco for interesting discussions. \'EF is partially supported by the  project ANR-16-CE40-0009-01 (GATO) and the project ANR-19-CE48-011-01 (COMBIN\'E).

\bibliographystyle{plain}
\bibliography{Allgenus}

\end{document}